\documentclass[12pt]{amsart}
\usepackage[all]{xy}
\usepackage[parfill]{parskip}

\usepackage{color}

\usepackage{amsmath, amscd, graphicx, latexsym, hyperref, times}
\usepackage{color,graphicx}
\oddsidemargin .25in \theoremstyle{plain}
\newtheorem{dummy}{anything}[section]
\newtheorem{theorem}[dummy]{Theorem}

\newtheorem{lemma}[dummy]{Lemma}

\newtheorem{question}[dummy]{Question}

\newtheorem{corollary}[dummy]{Corollary}

\theoremstyle{definition}
\newtheorem{definition}[dummy]{Definition}

\newtheorem{remark}[dummy]{Remark}

\newcommand{\del}{\partial}

\newcommand{\R}{\mathbb{R}}

\begin{document}

\title{canonical contact unit cotangent bundle}

\author{Takahiro Oba and Burak Ozbagci}

\address{Department of Mathematics, Tokyo Institute of Technology, 2-12-1 Ookayama, Meguroku, Tokyo
152-8551, Japan}
\email{oba.t.ac@m.titech.ac.jp}
\address{Department of Mathematics, Ko\c{c} University, Istanbul,
Turkey}
\email{bozbagci@ku.edu.tr}

\subjclass[2000]{}
\thanks{The first author was partially supported by JSPS KAKENHI Grant
Number 15J05214.}


\begin{abstract}

We describe an explicit open book decomposition adapted to the
canonical contact structure on the unit cotangent bundle of a closed
surface.

\end{abstract}

\maketitle

\section{Introduction}\label{sec: intro}

  Let $S$ denote a closed surface which is not necessarily orientable. Let  $\pi$ denote the projection
  of the bundle of cooriented lines tangent to $S$ onto $S$.
For a  point $q \in S$ and a cooriented line $u$ in $T_q S$, let
$\xi_{(q,u)}$ denote the cooriented
  plane described uniquely by the
  equation $\pi_*(\xi_{(q,u)}) = u \in T_q S$. The canonical contact structure $\xi_{can} $ on the bundle of cooriented lines
  tangent to $S$ consists of these planes (see, for example, \cite{ma}).

    If $S$ is equipped with a Riemannian
        metric, then the bundle of cooriented lines
  tangent to $S$ can be identified with the unit cotangent
        bundle $ST^*S$, and  $\xi_{can}$ is given by the kernel of
        the Liouville $1$-form $\lambda_{can}$ under this identification. Moreover, the disk
        cotangent bundle $DT^*S$ equipped with its canonical symplectic structure $\omega_{can}=d\lambda_{can}$ is a minimal strong symplectic
        filling of the contact $3$-manifold $(ST^*S, \xi_{can}) $.

        In this article, we describe an explicit abstract open book
        decomposition 
        adapted to the contact $3$-manifold $(ST^*S, \xi_{can})$, in the sense of Giroux
        \cite{g}. In the following, we use $\Sigma_g$ to
        denote the orientable closed surface
       of genus $g$ and $N_k$ to denote the non-orientable closed surface
       obtained by the connected sum of $k$ copies of the real
       projective plane $\mathbb{RP}^2$.

In Theorem~\ref{thm: ope} (resp. Theorem~\ref{theorem: g+2}), for
any $g \geq 1$, we describe an open book
        adapted to $(ST^*\Sigma_g, \xi_{can})$, whose page is a genus $g$ surface with $2g+2$ (resp. $g+2$)
        boundary components and we give an explicit factorization of its monodromy
        into a product of positive Dehn twists. In Corollary~\ref{cor: elf}, we also describe
        an exact symplectic Lefschetz fibration over $D^2$, whose total
        space is symplectomorphic to $(DT^*\Sigma_g, \omega_{can})$,
        up to completion.

        In Theorem~\ref{thm: ope2}, for
any $k \geq 1$, we describe an open book
        adapted to $(ST^*N_{k}, \xi_{can})$, whose page is a planar surface with $2k+2$
        boundary components and we give an explicit factorization of its monodromy
        into a product of positive Dehn twists.

        The  unit cotangent bundle $ST^*\Sigma_0$ is diffeomorphic to the real
        projective space $\mathbb{RP}^3$, and $\xi_{can}$ is the  unique tight contact
        structure in $\mathbb{RP}^3$, up to isotopy (cf. \cite{h}). It is well-known (see, for example \cite{eo}) that
        $(\mathbb{RP}^3, \xi_{can})$ has an adapted open book whose page is
        the annulus and whose monodromy is the square of the positive Dehn
        twist along the core circle of the annulus. Moreover, McDuff
        \cite{mc} showed that any minimal symplectic filling of
        $(\mathbb{RP}^3, \xi_{can})$ is diffeomorphic to $DT^*\Sigma_0$.

        The unit cotangent bundle $ST^*\Sigma_1$ is diffeomorphic to the
        $3$-torus $T^3$ and Eliashberg \cite{el} showed that $\xi_{can}$
        is the  unique strongly symplectically  fillable contact structure in $T^3$, up to contactomorphism. In his thesis \cite{vhm}, Van
        Horn-Morris constructed an explicit open book with genus one pages
        adapted to $(T^3, \xi_{can})$. Note that $(T^3, \xi_{can})$ can not
        be supported by a planar open book by a theorem of Etnyre \cite{et}.
        Moreover, according to Wendl \cite{w}, any minimal strong symplectic
        filling of $(T^3, \xi_{can})$ is symplectic deformation equivalent
        to $(DT^*\Sigma_1 \cong T^2 \times D^2, \omega_{can})$.

        The unit cotangent bundle $ST^*N_1$ is diffeomorphic to the
        lens space  $L(4,1)$ and $\xi_{can}$
        is the  unique universally tight contact structure in $L(4,1)$, up to
        contactomorphism. Note that the canonical contact structure on $L(4,1)$ viewed as a singularity link is isomorphic to
        $\xi_{can}$ defined as above.
       It is well-known (see, for example \cite{eo}) that  $(L(4,1),
\xi_{can})$ has
        an adapted open book whose page is
        the 4-holed sphere and whose monodromy is the product of positive Dehn
        twists along four curves each of which is parallel to a boundary component.  Moreover, McDuff
        \cite{mc} showed that $(L(4,1), \xi_{can})$ has two minimal symplectic fillings
        up to diffeomorphism: (i) the disk cotangent bundle $DT^*N_1$, which is a rational homology $4$-ball
        and (ii) the disk bundle over the sphere with Euler number $-4$.

We would like to point out that in an unpublished expository article
        \cite{m}, Massot argues that, for all $g \geq 2$,  the contact
        $3$-manifold $(ST^*\Sigma_g, \xi_{can})$ has an adapted open book
        with genus one pages but he does not describe its monodromy.
        The interested reader can turn to \cite{dh} and \cite{gy}
        for related material.

\section{Exact symplectic Lefschetz fibrations}\label{sec: eslf}
    Suppose that $W$ is a smooth $4$-manifold with nonempty boundary equipped with an exact
    symplectic form $\omega=d \alpha$ such that the Liouville vector field, which is by definition  $\omega$-dual to $\alpha$, is
    transverse to $\partial W$ and points outwards. Then $(W, \omega)$ is called an exact symplectic $4$-manifold with $\omega$-convex boundary
    and it is also called an exact symplectic filling of the contact $3$-manifold $(\partial W,
    \ker(\alpha|_{\partial W}))$ if the contact boundary is desired to be emphasized.

    The definition above can be extended to smooth manifolds with corners as follows (cf. \cite[Section 7a]{s}). Let
$W$ be a smooth $4$-manifold with codimension $2$ corners. An exact
symplectic structure on $W$ is given by a
   symplectic
    $2$-form $\omega= d\alpha$ on $W$ such that the Liouville vector field (again defined as $\omega$-dual to $\alpha$) is
    transverse to each boundary stratum of codimension $1$ and points outwards.
    It follows that $\alpha$ induces a contact form on each boundary
    stratum. Moreover, if the corners of $W$ are rounded off (see
\cite[Lemma 7.6]{s}), it becomes an exact symplectic filling of its
contact  boundary.

\begin{definition} \label{def:exactsym} An exact symplectic Lefschetz fibration on an exact symplectic $4$-manifold
$(W, \omega)$ with codimension $2$ corners is a smooth map $\pi :
(W, \omega)  \to D^2$ satisfying the following conditions:

\begin{itemize}

\item The map $\pi$ has finitely many critical points $p_1, \ldots, p_k$ in the interior of $W$ such
that around each critical point,  $\pi$ is modeled on the map $(z_1,
z_2) \to z_1^2 + z_2^2$ in complex local coordinates compatible with
the orientations.

\item Every fiber of the map $\pi|_{W \setminus \{p_1,
\ldots, p_k\}} : W \setminus \{p_1, \ldots, p_k\} \to D^2$ is a
symplectic submanifold.

\item $\partial W$ consists of two smooth boundary strata $\partial_vW$ (the vertical boundary) and
$\partial_hW$(the horizontal boundary) meeting at a codimension $2$
corner, where $$\partial_vW = \pi^{-1} (\partial D^2)\; \mbox{and}\;
\partial_hW = \bigcup_{z \in D^2} \partial (\pi^{-1} (z) ).$$  We require that
$\pi|_{\partial_vW} :
\partial_vW \to
\partial D^2$ is smooth fibration and $\pi$ is a trivial smooth fibration over $D^2$ near
$\partial_hW$.

\end{itemize}

\end{definition}

The vertical boundary $\partial_vW$ is a surface fibration over the
circle and the horizontal boundary $\partial_hW$ is a disjoint union
of some number of copies of $S^1 \times D^2$. The vertical and
horizontal boundaries  meet each other at the corner $$\partial_vW
\cap
\partial_hW =
\partial(\partial_hW)= \coprod (S^1 \times \partial D^2).$$ Therefore, after
rounding off the corners of $W$, its boundary $\partial W$ acquires
an open book decomposition given by $\pi|_{\partial W \setminus B} :
\partial W \setminus B \to \partial D^2$, where $\partial_hW$ is
viewed as a tubular neighborhood of the binding $B: = \coprod (S^1
\times \{0\})$. Moreover, $\alpha$ restricts to a contact form on
$\partial W$ whose kernel is a contact structure supported by this
open book.

\begin{remark} A \emph{smooth} Lefschetz fibration on a smooth $4$-manifold $W$ with codimension $2$ corners,
is a smooth map  $\pi : W  \to D^2$  which satisfies the first and
the last conditions listed in the Definition~\ref{def:exactsym}.
\end{remark}

Next, we briefly recall (see \cite[Section 16]{s}, \cite[Chapter
8]{gs}) how the topology of the total space of an exact symplectic
Lefschetz fibration
$$\pi : (W, \omega) \to D^2$$ is described using a distinguished basis of
vanishing paths in $D^2$.

Without loss of generality, we can assume that $D^2$ is the unit
disk $\mathbb{D}$ in $\mathbb{C}$. For each critical value $z \in
\mathbb{D}$ of the fibration $\pi$, the fiber $\pi^{-1}(z)$ is
called a \emph{singular} fiber, while the other fibers are called
\emph{regular}. Throughout this paper, we will assume that a regular
fiber is connected and each singular fiber contains a unique
critical point. By setting $z_0=1 \in
\partial \mathbb{D}$, the regular fiber $F=\pi^{-1} (z_0)$, which is a symplectic submanifold of $(W, \omega)$,
serves as a reference fiber in the discussion below.  We call
$z_0$ the base point.

For any critical value  $z \in int \; \mathbb{D}$,  a vanishing path
is an embedded path $\gamma : [0,1] \to \mathbb{D}$ such that
$\gamma(0)=z_0$ and $\gamma(1)=z \in \mathbb{D}$. To each such path,
one can associate its Lefschetz thimble $\Delta_{\gamma}$, which the
unique embedded Lagrangian disk in $(W, \omega)$ such that $\pi
(\Delta_{\gamma}) = \gamma([0,1])$ and $\pi (\partial
\Delta_{\gamma}) = z_0$. The boundary $\partial \Delta_{\gamma}$ of
the Lefschetz thimble is therefore an (exact) Lagrangian circle in
$(F, \omega|_F)$. This circle is called a vanishing cycle since
under a parallel transport along $\gamma$, it collapses to the
unique singular point on the fiber $\pi^{-1} (z)$.

A distinguished basis of vanishing paths is an ordered set of
vanishing paths $(\gamma_1, \ldots, \gamma_k)$ (one for each
critical value of $\pi$) starting at the base point $z_0$ and ending
at a critical value such that   $\gamma_i$ intersects $\gamma_j$
only at $z_0$ for $i \neq j$. Note that  there is a natural
counterclockwise ordering of these paths, by assuming that the
starting directions of the paths are pairwise distinct. Let
$\delta_i$ denote the vanishing cycle in $F$ corresponding to the
vanishing path $\gamma_i$, whose end point---a critical value---is
labeled as $z_i$.

Now consider a small loop, oriented counterclockwise, around the
critical value $z_i$, and connect it to the base point $z_0$ using
the vanishing path $\gamma_i$.  One can consider this loop as a loop
$a_i$ around $z_i$ passing through $z_0$ and not including any other
critical values in its interior. It is a classical fact that
$\pi^{-1}(a_i)$ is a
 surface bundle over $a_i$, which is diffeomorphic to
$$(F \times[0,1])/((x, 1) \sim (D(\delta_i)(x), 0)$$ where $D(\delta_i)$
denotes the positive Dehn twist along the vanishing cycle $\delta_i
\subset F$.

Similarly,  $\pi^{-1} (\partial \mathbb{D})$ is an $F$-bundle over
$\mathbb{S}^1=
\partial \mathbb{D}$ which is diffeomorphic to $$(F \times[0,1])/((x, 1) \sim
(\psi(x), 0)$$ for some self-diffeomorphism $\psi$ of the fiber $F$
preserving $\partial F$ pointwise. The map $\psi$ is called the
geometric monodromy and computed via parallel transport using any
choice of a  connection on the bundle. Note that the isotopy class
of $\psi$ is independent of the choice of the connection.  It
follows that
$$ \psi = D(\delta_k) D(\delta_{k-1}) \cdots D(\delta_1) \in
Map(F, \partial F),$$  where $Map(F, \partial F)$ denotes the
mapping class group of the surface $F$. The product of positive Dehn
twists above is called a \emph{monodromy factorization} or a
\emph{positive factorization} of the monodromy $\psi$ of the
Lefschetz fibration $\pi$.

Note that the vanishing cycle for each singular fiber is determined
by the choice of a  vanishing path ending at the corresponding
critical value. Therefore a different basis of vanishing paths (with
the same rules imposed as above) induce a different factorization of
the monodromy $\psi$. Nevertheless, any two distinguished bases of
vanishing paths are related by a sequence of transformations---the
\emph{Hurwitz moves}. An elementary Hurwitz move is obtained by
switching the order of two consecutive vanishing paths as shown in
Figure~\ref{fig: Hurwitz} keeping the other vanishing paths fixed.
This will have the following affect on the ordered set of vanishing
cycles
$$(\delta_1,\ldots, \delta_{i-1}, \delta_i, \delta_{i+1},
\delta_{i+2}, \ldots, \delta_k) \to (\delta_1,\ldots, \delta_{i-1},
\delta_{i+1}, D(\delta_{i+1})(\delta_i), \delta_{i+2}, \ldots,
\delta_k),$$ which is also called an elementary Hurwitz move. In
general a Hurwitz move is any composition of elementary Hurwitz
moves and their inverses.

    \begin{figure}[h]
                    \begin{center}
                        \includegraphics[width=280pt]{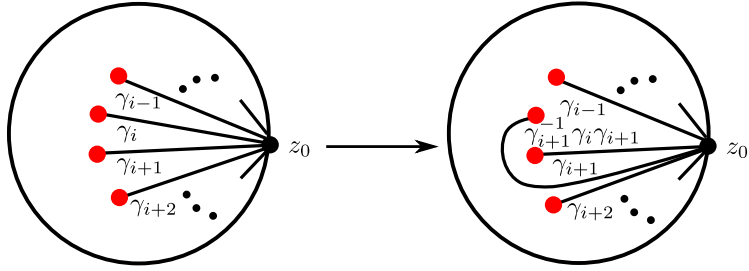}
                        \caption{An elementary Hurwitz move}
                        \label{fig: Hurwitz}
                    \end{center}
    \end{figure}

If one chooses a different base point on $\partial \mathbb{D}$ to
begin with, then the monodromy of the Lefschetz fibration takes the
form $\varphi \psi \varphi^{-1}$, where $\varphi$ is the appropriate
element of $Map(F, \partial F)$, obtained by parallel transport. In
this case, the monodromy factorization appears as
\begin{align*}
& \varphi \psi \varphi^{-1} = \varphi \big(D(\delta_k) D(\delta_{k-1}) \cdots D(\delta_1)\big) \varphi^{-1}\\
& =\varphi D(\delta_k) \varphi^{-1}\varphi
D(\delta_{k-1})\varphi^{-1}\varphi \cdots \varphi^{-1}\varphi D(\delta_1) \varphi^{-1} \\
& =D(\varphi(\delta_k))D(\varphi(\delta_{k-1})) \cdots
D(\varphi(\delta_1)),
\end{align*}
where the last equality follows by the fact that the conjugation
$\varphi D(\delta) \varphi^{-1}$ of a positive Dehn twist
$D(\delta)$ is isotopic to the positive Dehn twist
 $D(\varphi(\delta))$.

Conversely (cf. \cite[Lemma 16.9]{s}),

\begin{lemma}\label{lem: cons} Let $(\delta_1, \ldots, \delta_k)$ be an ordered collection  of
embedded (Lagrangian) circles on an exact symplectic surface $F$
with nonempty boundary.
 Choose a base point $*$ on
$\partial D^2$, and a distinguished basis of vanishing paths
$(\gamma_1, \ldots, \gamma_k)$ starting at $*$. Then there is an
exact symplectic Lefschetz fibration $\pi : (W,\omega) \to D^2$
whose critical values are $\gamma_1(1), \ldots, \gamma_k (1)$, under
which $\delta_i$ corresponds to the vanishing cycle for the path
$\gamma_i$, where $\pi^{-1} (*)= F$ as symplectic manifolds.
Moreover, this fibration is trivial near $\partial_hW$.
\end{lemma}

\begin{definition} A conformal exact
symplectomorphism between two exact symplectic $4$-manifolds $(W_1,
\omega_1=d\alpha_1)$ and  $(W_1, \omega_2=d\alpha_2)$ is a
diffeomorphism $\phi : W_1 \to W_2$ such that $\phi^*\alpha_2 =
K\alpha_1 + df$ for some smooth function $f: W_1 \to \mathbb{R}$,
and some real number $K>0$. If $K=1$, then $\phi$ is called an exact
symplectomorphism.
\end{definition}

\begin{remark} The definition above also applies to maps between exact symplectic $4$-manifolds with codimension $2$ corners. \end{remark}

\begin{lemma} \label{lem: hur} Suppose that $\pi: (W, \omega) \to D^2$ is an exact symplectic Lefschetz
fibration whose ordered set of vanishing cycles is given by $$
(\delta_1,\ldots, \delta_{i-1}, \delta_i, \delta_{i+1},
\delta_{i+2}, \ldots, \delta_k)$$ with respect to some distinguished
basis of vanishing paths. Then there is an exact symplectic
Lefschetz fibration $\widetilde{\pi} : (\widetilde{W},
\widetilde{\omega}) \to D^2$ whose ordered set of vanishing cycles
is given by $$ (\delta_1,\ldots, \delta_{i-1}, \delta_{i+1},
D(\delta_{i+1})(\delta_i), \delta_{i+2}, \ldots, \delta_k)$$ with
respect to some distinguished basis of vanishing paths, such that
$\pi$ and $\widetilde{\pi}$ are isomorphic through an exact
symplectomorphism $\phi: (W, \omega) \to (\widetilde{W},
\widetilde{\omega})$.
\end{lemma}

\begin{proof} By Lemma~\ref{lem: cons}, there is an exact symplectic Lefschetz fibration $\widetilde{\pi}
: (\widetilde{W}, \widetilde{\omega}) \to D^2$  whose ordered set of
vanishing cycles is given by $$ (\delta_1,\ldots, \delta_{i-1},
\delta_{i+1}, D(\delta_{i+1})(\delta_i), \delta_{i+2}, \ldots,
\delta_k)$$ \emph{with respect to the distinguished basis of
vanishing paths of $\pi$}. Now we apply an elementary inverse
Hurwitz move on this distinguished basis of vanishing paths of
$\widetilde{\pi}$, so that the ordered set of vanishing cycles of
$\widetilde{\pi}$ agrees, up to isotopy, with the ordered set of
vanishing cycles of $\pi$. Note that we keep the fibration
$\widetilde{\pi}$ fixed, while modifying its distinguished basis of
vanishing paths.

It follows that $\pi$ and $\widetilde{\pi}$ are two exact symplectic
Lefschetz fibrations whose ordered set of vanishing cycles are
isotopic. The result follows by the fact that an exact symplectic
Lefschetz fibration is uniquely determined---up to isomorphism via
an exact symplectomorphism of its total space---by its regular fiber
and the isotopy class of its ordered set of vanishing cycles.
\end{proof}

It is well-known that a positive stabilization of a smooth Lefschetz
fibration is a smooth Lefschetz fibration. In the following we
briefly explain positive stabilizations of exact symplectic
Lefschetz fibrations (cf. \cite[Appendix A]{mcl}).

 A positive stabilization of an exact symplectic Lefschetz
fibration $\pi : (W, \omega) \to D^2$ along a properly embedded
(Lagrangian) arc $\beta$ in $(F, \omega|_F)$, where $F$ is the
reference regular fiber of $\pi$ as above, is another an exact
symplectic Lefschetz fibration $\pi'' : (W'', \omega'') \to D^2$
defined as follows.

First, we attach a $4$-dimensional Weinstein $1$-handle to $(W,
\omega)$ along the two endpoints of $\beta \subset F \subset
\partial W$ such that $\omega$ extends over the $1$-handle as an exact
symplectic form $\omega'=d\alpha'$ to obtain an exact symplectic
Lefschetz fibration $\pi': (W', \omega') \to D^2$ which agrees with
$\pi$ when restricted to $(W, \omega)$. In order to see this, we
view the $4$-dimensional Weinstein $1$-handle as a thickening $D^2
\times H^2_1$ of the $2$-dimensional Weinstein $1$-handle $H^2_1$,
where $D^2$ is the base disk of the fibration equipped with the
standard symplectic structure. In other words, we extend each fiber
of $\pi$ by attaching a Weinstein $1$-handle $H_1^2$, so that the
exact symplectic form $\omega$ extends fiberwise.  In particular,
the reference regular fiber $F'= F \cup H_1^2$ of $\pi'$ is obtained
by attaching  $H_1^2$ to $F$ along the endpoints of $\beta$. Let
$\beta' \subset (F', \omega'|_{F'})$ denote the closed Lagrangian
curve obtained from $\beta$  by gluing in the core circle of
$H_1^2$.

Next, we attach a $4$-dimensional $2$-handle to $(W', \omega')$
along the curve $\beta' \subset F' \subset
\partial W'$ with framing $-1$ relative to its fiber framing. It is a classical fact (\cite[Section 8.2]{gs}) that
the result is a smooth Lefschetz fibration $\pi'' : W'' \to D^2$,
which has one more critical point (with vanishing cycle $\beta'$) in
addition to those of $\pi': (W', \omega') \to D^2$.  Moreover, by
the Legendrian Realization Principle \cite{h},  $\beta'$ can be
realized as a Legendrian curve on $F'$ in the contact boundary
$(\partial W', \ker (\alpha'))$ so that its contact framing agrees
with its fiber framing. It follows that the aforementioned
``Lefschetz" $2$-handle can be considered as a Weinstein $2$-handle
(see, for example, \cite[Section 7.2]{ozst})  and hence $W''$ admits
an exact symplectic form  $\omega''$ which restricts to $\omega'$ on
$W'$.

All we have to argue now is that $\omega''$ restricts to a
symplectic structure on the fibers of the smooth Lefschetz fibration
$\pi''$. To see this, we consider the standard local model (see,
\cite[Example 15.4]{s}) around a critical point in an exact
symplectic Lefschetz fibration, where a regular fiber is
symplectomorphic to the disk cotangent bundle $DT^*S^1$ of a circle,
equipped with its canonical symplectic structure $\lambda_{can}$. We
also note that around the new critical point at the origin of the
model Weinstein $2$-handle, the smooth Lefschetz fibration agrees
\emph{smoothly} with the standard local model of an exact symplectic
Lefschetz fibration. But, since in both models we use the standard
symplectic structure on $\mathbb{R}^4 = \mathbb{C}^2$, the smooth
Lefschetz fibration can be simply viewed as an exact symplectic
Lefschetz fibration.

The point is that the fibers of $\pi'$ is already symplectic and by
attaching the Weinstein $2$-handle along the Lagrangian curve
$\beta'$ in a symplectic fiber $(F', \omega'|_{F'})$ on the
boundary,  we identify a symplectic neighborhood of $\beta' \subset
(F', \omega'|_{F'})$ with
 $(DT^*S^1, \lambda_{can})$ by the Lagrangian neighborhood theorem.  As a matter of fact,
 the fibers
 of $\pi'$ and $\pi''$ are symplectomorphic, where the monodromy of
 $\pi''$ is obtained by composing the monodromy of
 $\pi'$ by a symplectic Dehn twist around $\beta'$.

Finally, since the attaching sphere of the $2$-handle intersects the
belt sphere of the $1$-handle at a unique point, these two handles
cancel each other out smoothly. Moreover, this cancelation also
takes place in the symplectic category, up to completion, by a
theorem of Eliashberg \cite[Lemma 3.6b]{el2} (see also \cite{ce} or
\cite[Lemma 3.9]{vk}).

The discussion above can be summarized as follows.

\begin{lemma} \label{lem: stab} Any positive stabilization of an exact symplectic Lefschetz fibration
is an exact symplectic Lefschetz fibration. Moreover,
if $\pi'': (W'', \omega'') \to D^2$ is a positive stabilization of
an exact symplectic Lefschetz fibration $\pi : (W, \omega) \to D^2$,
then $(W'', \omega'')$ and $(W, \omega)$ have symplectomorphic
completions.
\end{lemma}

Moreover, the open book on $\partial W''$ induced by $\pi''$ is
obtained by a positive stabilization of the open book on $\partial
W$ induced by $\pi$, by definition. Therefore the contact manifold
$(\partial W'', \ker(\alpha''))$ is contactomorphic to the contact
manifold $(\partial W, \ker(\alpha))$, where $\omega''=d \alpha''$
and $\omega=d \alpha$.

\section{Explicit open book decompositions adapted to the unit contact cotangent
bundle} \label{sec: exp}

For any closed surface $S$,   Johns \cite{Jo} constructed an exact
symplectic Lefschetz fibration $\pi: (E, \omega) \rightarrow D^{2}$
such that $(E, \omega)$ is
    conformally exact symplectomorphic to the disk cotangent bundle $DT^{*}S$ equipped
    with its canonical symplectic form
$\omega_{can}$. In the following we give a brief summary of Johns'
work.

   Johns' initial idea was to try to ``complexify"  a Morse function $f : S \to \mathbb{R}$
    in order to find a Lefschetz fibration $\pi : DT^{*}S \to \mathbb{C}$, generalizing the work of A'Campo
    \cite{ac}. Since, this method turned out to be difficult, he
    took a different approach instead.

Modifying a simple construction of a Lefschetz pencil on
    $\mathbb{CP}^2$ discussed in \cite[Section 5.2]{as}, Johns first worked out the case of
    $S=\mathbb{RP}^2$ obtaining a Lefschetz fibration $\pi : DT^{*} \mathbb{RP}^2 \to
    D^2$ with three vanishing cycles explicitly described on the fiber, a
    $4$-holed sphere. The key point in his construction is that the
    map
    $\pi$ restricted to the standard embedding of $\mathbb{RP}^2$
    into $\mathbb{CP}^2$ is the standard Morse function on $\mathbb{RP}^2$
    with three critical points.

    As a second example, Johns worked out the case $S=T^2$. Starting from a Lefschetz fibration
    $\mathbb{C}^* \times \mathbb{C}^* \to \mathbb{C}$,
    he obtained a Lefschetz fibration $\pi : DT^{*} T^2 \to
    D^2$ arising from the embedding $T^2 = S^1 \times S^1 \subset \mathbb{C}^* \times
    \mathbb{C}^*$. Again, he showed that $\pi$ restricted to $T^2$
    is a Morse function on $T^2$ with four critical points.  The
    regular fiber of the Lefschetz fibration $\pi$ in this case is a
    $4$-holed torus, although Johns did not explicitly describe the
    four vanishing cycles.

Nevertheless, based on the pattern occurring in these basic
examples, Johns was able to have an educated guess on how the fiber
and the vanishing cycles would look like for a Lefschetz fibration
on $DT^{*}S$ for a general compact surface $S$ without boundary.

Starting with a Morse function $f: S \to \mathbb{R}$ with one
minimum, one  maximum and $m$ index $1$ critical points, Johns
constructed a Lefschetz fibration $\pi: E \rightarrow D^{2}$ by
describing its regular fiber, a \emph{necessarily orientable}
surface $F$ obtained from the annulus by attaching $2m$ one handles,
and giving explicitly  the set of vanishing cycles consisting of
$m+2$ simple closed curves on $F$.

Here the annulus $S^1 \times [-1,1]$ can be viewed as the disk
cotangent bundle $DT^*V_0$, where $V_0$ is the vanishing cycle
corresponding to the minimum of $f$. For each index $1$ critical
point of $f$, two $1$-handles are attached to the annulus and the
attachment of  these $1$-handles can be viewed as the plumbing of
$DT^*V_0$ with another disk cotangent bundle  $DT^*V^i_1$, where
$V^i_1$ denotes the vanishing cycle corresponding to that index $1$
critical point. There are two kinds of plumbing descriptions,
however, depending on whether the index $1$ critical point of $f$
induces an orientable or a non-orientable $1$-handle in the handle
decomposition of $S$.

   \begin{figure}[h]
                    \begin{center}
                        \includegraphics[width=350pt]{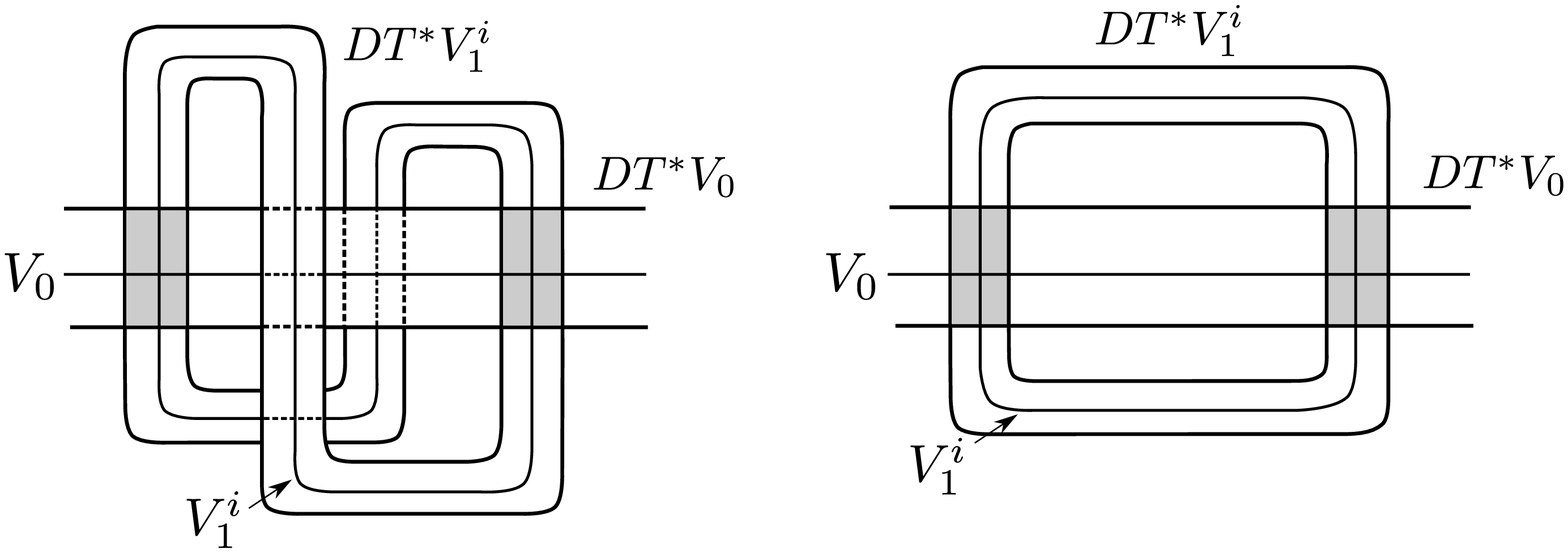}
                        \caption{Plumbing description: orientable case on the left, non-orientable case on the right }
                        \label{fig: plumbing}
                    \end{center}
    \end{figure}

Finally, there is one last vanishing cycle $V_2$, corresponding to
the maximum of $f$, obtained by the Lagrangian surgery (see
Figure~\ref{fig:  Lagrangian surgery}, for an example) of $V_0$ with
the union $\bigcup_{i=1}^{m} V^1_i$.

 \begin{figure}[h]
                    \begin{center}
                        \includegraphics[width=250pt]{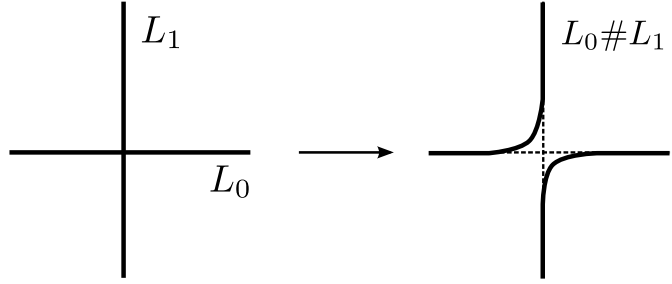}
                        \caption{Left: Two Lagrangian curves $L_0$ and $L_1$ on a surface intersecting locally as shown. Right:
                        Lagrangian surgery $L_0 \# L_1$}
                        \label{fig: Lagrangian surgery}
                    \end{center}
    \end{figure}
By the discussion in Section~\ref{sec: eslf}, the $4$-manifold $E$
admits an exact symplectic form $\omega$, for which $\pi : E \to
D^2$ is an exact symplectic Lefschetz fibration. Moreover, Johns
verified that

\begin{itemize}

\item $S$ admits an exact Lagrangian embedding into $E$.

\item The critical points of $\pi$ lie on $S$, $\pi(S) \subset
\mathbb{R}$, and $\pi|_S= f$.

\item The symplectic manifold $(E, \omega)$,  after smoothing the
corners of $E$, is conformally exact symplectomorphic to $(DT^{*}S,
\omega_{can})$.

\end{itemize}

Apparently, the most difficult step is the first item above. In
order to find such an embedding Johns used a ``Milnor-type" handle
decomposition---a more refined version of a usual handle
decomposition---of the surface $S$, referring to \cite[pages
27-32]{mi}.  Once this is achieved, the second item follows from the
first by construction. The last item is essentially a retraction of
$E$, by a Liouville type flow,  onto a small Weinstein neighborhood
of $S$, which is symplectomorphic to $(DT^*S, \omega_{can})$.

In order to prove the main results of our article, we focus on the
orientable surface case in Section~\ref{sec: ori}, while in
Section~\ref{sec: nori}, we treat the non-orientable surface case.
For both cases, we use a handle decomposition of a closed surface
induced by the standard Morse function with one minimum and one
maximum, although this assumption can be removed as pointed out in
\cite[Section 4.3]{Jo}.

\subsection{Unit contact cotangent bundles of orientable
surfaces}\label{sec: ori} In this section, we assume that $S$ is a
closed \emph{orientable} surface of genus $g$, which we denote by
$\Sigma_g$. We also denote the exact symplectic Lefschetz fibration of Johns
described  above by $\pi_g: (W_g, \omega_g) \rightarrow D^{2}$,
where $(W_g, \omega_g)$ is
    conformally exact symplectomorphic to $(DT^{*}\Sigma_{g},
    \omega_{can})$. We first review the Lefschetz fibration $\pi_g$, primarily
    focusing on its topological aspects.

    The regular fiber $F_g$ of $\pi_g$ is diffeomorphic to an  oriented genus $g$ surface with $2g+2$ boundary components.
    In the following, we describe the construction of $F_g$, referring
    to Figure~\ref{fig: Rg}.

\begin{figure}[h]
                    \begin{center}
                        \includegraphics[width=350pt]{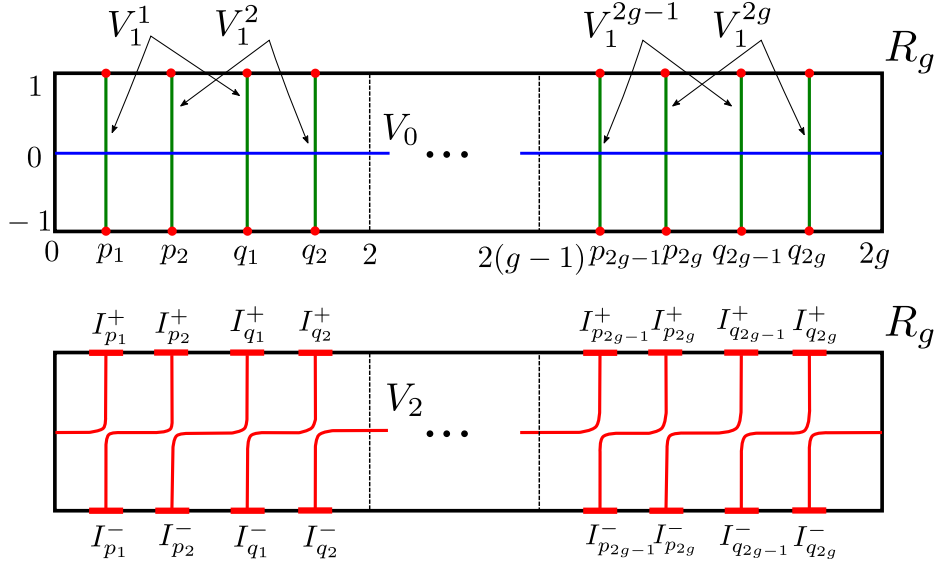}
                        \caption{The vanishing cycles $V_{0}, V_{1}^{1}, \dots, V_{1}^{g},
                        V_{2}$ }
                        \label{fig: Rg}
                    \end{center}
    \end{figure}

    Let $R_{g}$ denote the rectangle $[0,2g] \times [-1,1]$ in $ \R^{2}$ equipped with the standard
    orientation. We fix the following points  $$p_{2i-1} := 2(i-1)+ 1/3,\;  p_{2i} :=
2(i-1)+2/3,$$ $$q_{2i-1}:= 2(i-1)+4/3, \; q_{2i} := 2(i-1)+5/3,$$
    for $i=1,2,\dots, g$, on the $x$-axis. For a sufficiently small $\varepsilon >0$, we set
$$I^{\pm}_{p_{j}}:= [p_{j}-\varepsilon, p_{j} + \varepsilon] \times \{ \pm 1\} ,
    \ I^{\pm}_{q_{j}} := [q_{j}-\varepsilon, q_{j}+\varepsilon] \times \{ \pm 1\} \subset R_{g}$$
for $j=1,2, \dots,
    2g$. We identify $\{0\} \times [-1,1] \subset R_{g}$ with $\{2g \} \times
[-1,1] \subset R_{g}$ to obtain an annulus initially. Next, for each
$j=1,2, \ldots,
    2g$, we identify $I_{p_{j}}^{+}$ with $I_{q_{j}}^{-}$, and
$I_{p_{j}}^{-}$ with $I_{q_{j}}^{+}$ so that
    the orientation of $R_{g}$ extends to the resulting surface
    $F_g$.

    Note that, for each $j=1,2, \ldots,
    2g$, these identifications can be viewed as attaching two
    $1$-handles,  which is the same as plumbing an annulus as shown
    on the left in Figure~\ref{fig: plumbing}.

    By calculating the Euler characteristic, for example,
it can be easily seen that  $F_g$ is diffeomorphic to an  oriented
genus $g$ surface with $2g+2$ boundary components.

   By fixing a certain choice of distinguished set of vanishing
     paths, the $2g+2$ vanishing cycles $V_{0}, V_{1}^{1}, V_{1}^{2}, \dots, V_{1}^{2g},
     V_{2}$ of the Lefschetz fibration $\pi_g$,
    are given as follows. The vanishing cycle
    $V_{0}$ is the simple closed curve in $F_g$ obtained from $[0, 2g] \times \{0\} \subset R_{g}$ through the above identifications.
     Similarly, the simple closed curve $V_{1}^{j} \subset F_g$ is obtained from $(\{ p_{j}\} \times [-1,1]) \cup (\{q_{j}\} \times
     [-1,1]) \subset R_{g}$. Equivalently, $V_{1}^{j}$ is the core circle of the
     annulus that appears in the plumbing description (see Figure~\ref{fig: plumbing}). The vanishing cycle $V_{2} \subset F_g$ comes from the Lagrangian
surgery of $V_{0}$ and $\cup_{i=1}^{2g}V_{1}^{j}$ as
    depicted at the bottom of Figure \ref{fig: Rg}.

    Next we show that  the vanishing cycles $V_{0}, V_{1}^{1}, V_{1}^{2}, \dots, V_{1}^{2g},
     V_{2}$ can be presented with a different point of view, by reconstructing
     $F_g$ as follows. Let
    $$A_{i}:= [2(i-1), 2i] \times [-1,1] \subset R_{g}$$
    for $i=1,2,\dots,g$ and let
    $$J_{j} := \{ j \} \times [-1,1]$$
    for $j=0,1,\dots, 2g$. It is clear that $R_{g} =
    \cup_{i=1}^{g}A_{i}$. Then we divide each $A_{i}$ into two pieces $$A_{i}^{+}:= [2(i-1),
2i-1] \times [-1,1], \; \ A_{i}^{-}:=[2i-1,2i] \times [-1,1]$$ and
    put $A_{i}^{+}$ vertically on top of $A_{i}^{-}$ as shown in Figure \ref{fig: identification} (a).

    Since $I^{\pm}_{p_{2i-1}}$ and $I^{\pm}_{p_{2i}}$ belong to $A_{i}^{+}$
and $I^{\pm}_{q_{2i-1}}$ and $I^{\pm}_{q_{2i}}$ belong to
$A_{i}^{-}$,
    we can glue $A_{i}^{+} $ and $A_{i}^{-}$ along these intervals.
    Each of these gluings is
    represented by a $1$-handle in Figure \ref{fig: identification} (b). Moreover,
    we identify $J_{2i-1} \subset A_{i}^{+}$ with $J_{2i-1}
\subset A_{i}^{-}$, which is also represented by a $1$-handle. There
is another $1$-handle associated to the identification of $J_{2i}
\subset A_{i}^{-}$ with $J_{2i} \subset  A_{i+1}^{+}$. We slide this
$1$-handle over the one coming from the identification of
$I_{p_{2i}}^{+}$ with $I_{q_{2i}}^{-}$ as indicated in Figure \ref{fig: identification} (c).

    Starting from the diagram in Figure \ref{fig: identification} (c), and performing isotopies as shown in
 Figure \ref{fig: isotopy}, we now obtain a genus $1$ surface with
$3$ boundary components.
    We call this surface the ``building block", and  denote it by
    $F_g^{i}$. The key point is that the surface $F_g$ can be
    constructed by assembling these building blocks, which looks pairwise identical.

\begin{figure}[h]
                    \begin{center}
                        \includegraphics[width=315pt]{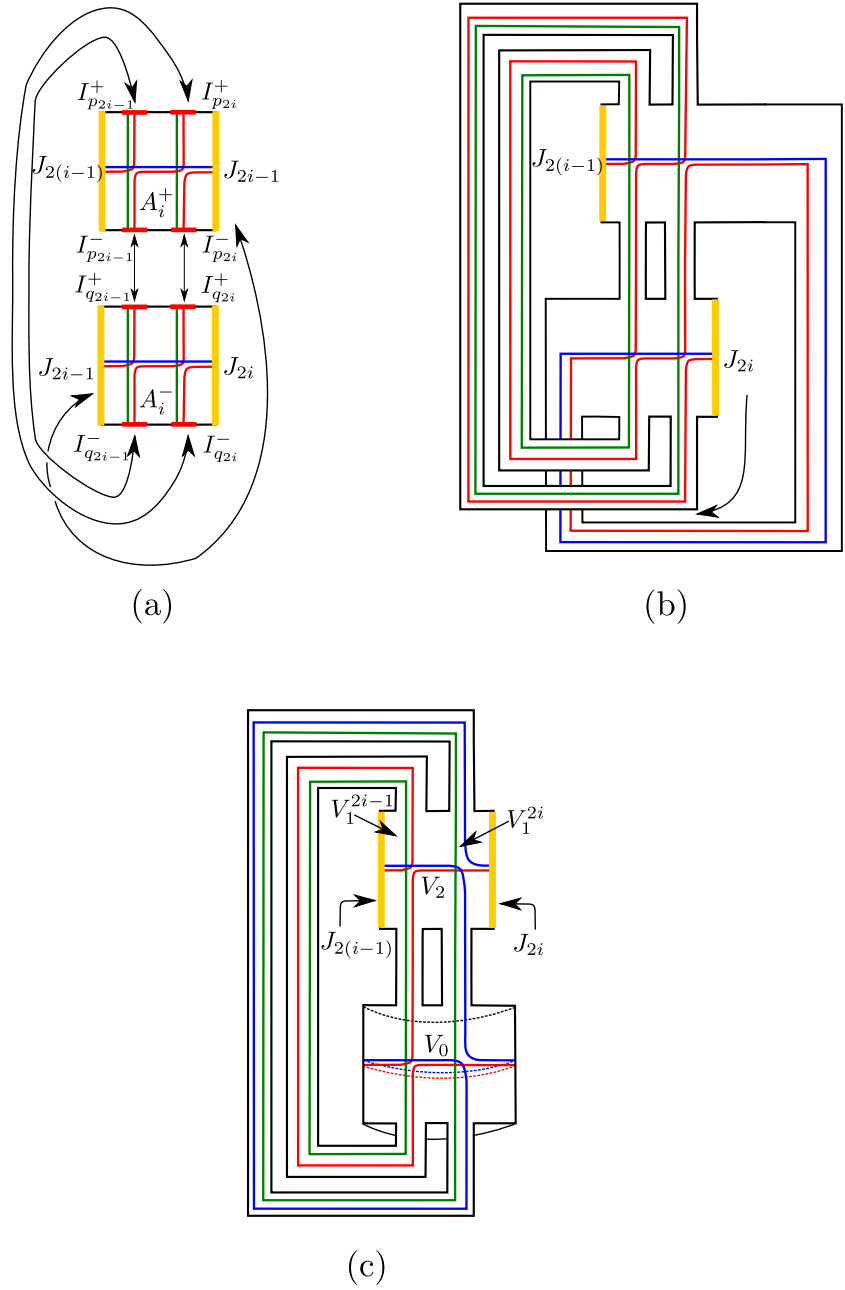}
                        \caption{The building block $F_g^i$}
                        \label{fig: identification}
                    \end{center}
    \end{figure}
    Note that the vanishing cycles can also be isotoped through the identifications and isotopies described above.
As a result, in each $F_g^{i}$, we see two arcs $\tilde{V}_{0}^{i}$
and  $\tilde{V}_{2}^{i}$ which are subarcs of $V_{0}$ and $V_{2}$,
respectively. We also see two simple closed curves $V_{1}^{2i-1}$
and $V_{1}^{2i}$, as depicted on the right in Figure \ref{fig:
isotopy}.
    Finally, to describe $F_g$ and the vanishing cycles $V_{0}, V_{1}^{1}, \dots, V_{1}^{2g},
    V_{2}$, we
    arrange $F_g^{1}, \dots, F_g^{g}$ in a circular position, glue $F_g^{i}$ to $F_g^{i+1}$ along $J_{2i}$ for $i=1,\dots, g-1$
    and glue $F_{g}^g$ to $F_g^{1}$ along $J_{2g}$ and $J_{0}$ as shown
    in  Figure \ref{fig: vanishing cycle}.

    Since the fiber $F_g$ is a genus
    $g$ surface with $2g+2$ boundary components, we opted to denote it with
    $\Sigma_{g,2g+2}$ in Theorem~\ref{thm: ope}.

    \begin{figure}[h]
                    \begin{center}
                        \includegraphics[width=460pt]{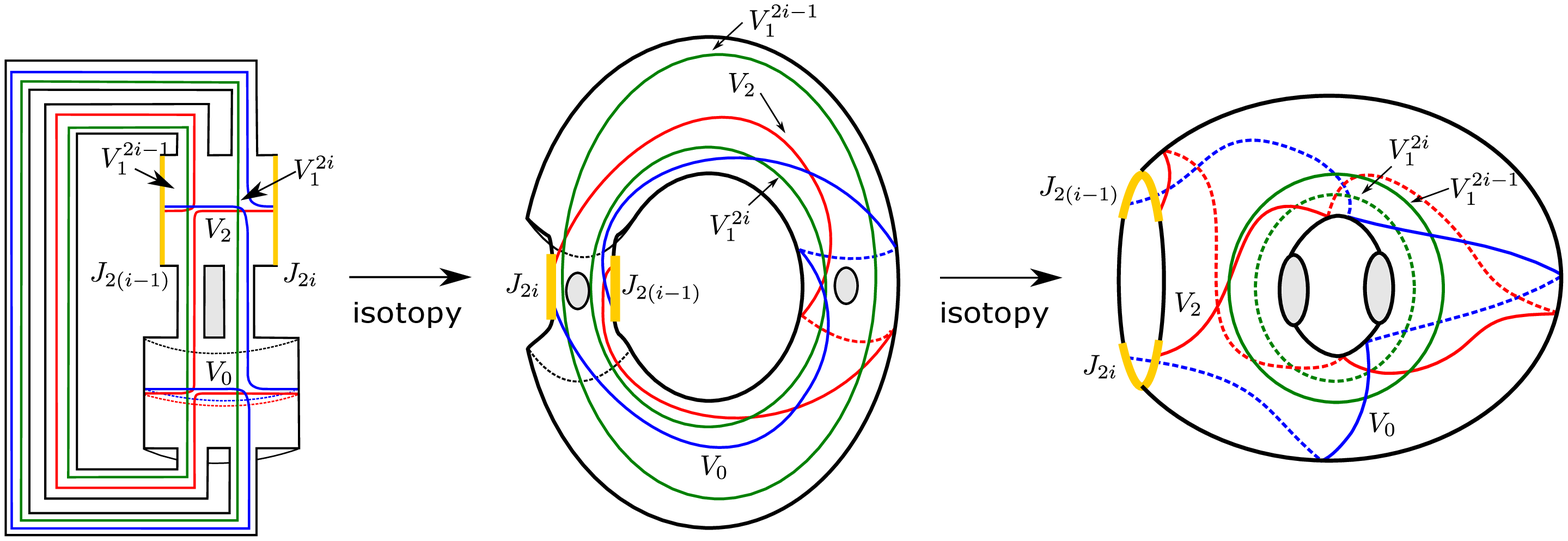}
                        \caption{Isotoping the vanishing cycles on $F_g^i$}
                        \label{fig: isotopy}
                    \end{center}
    \end{figure}
  \begin{figure}[h]
                    \begin{center}
                        \includegraphics[width=460pt]{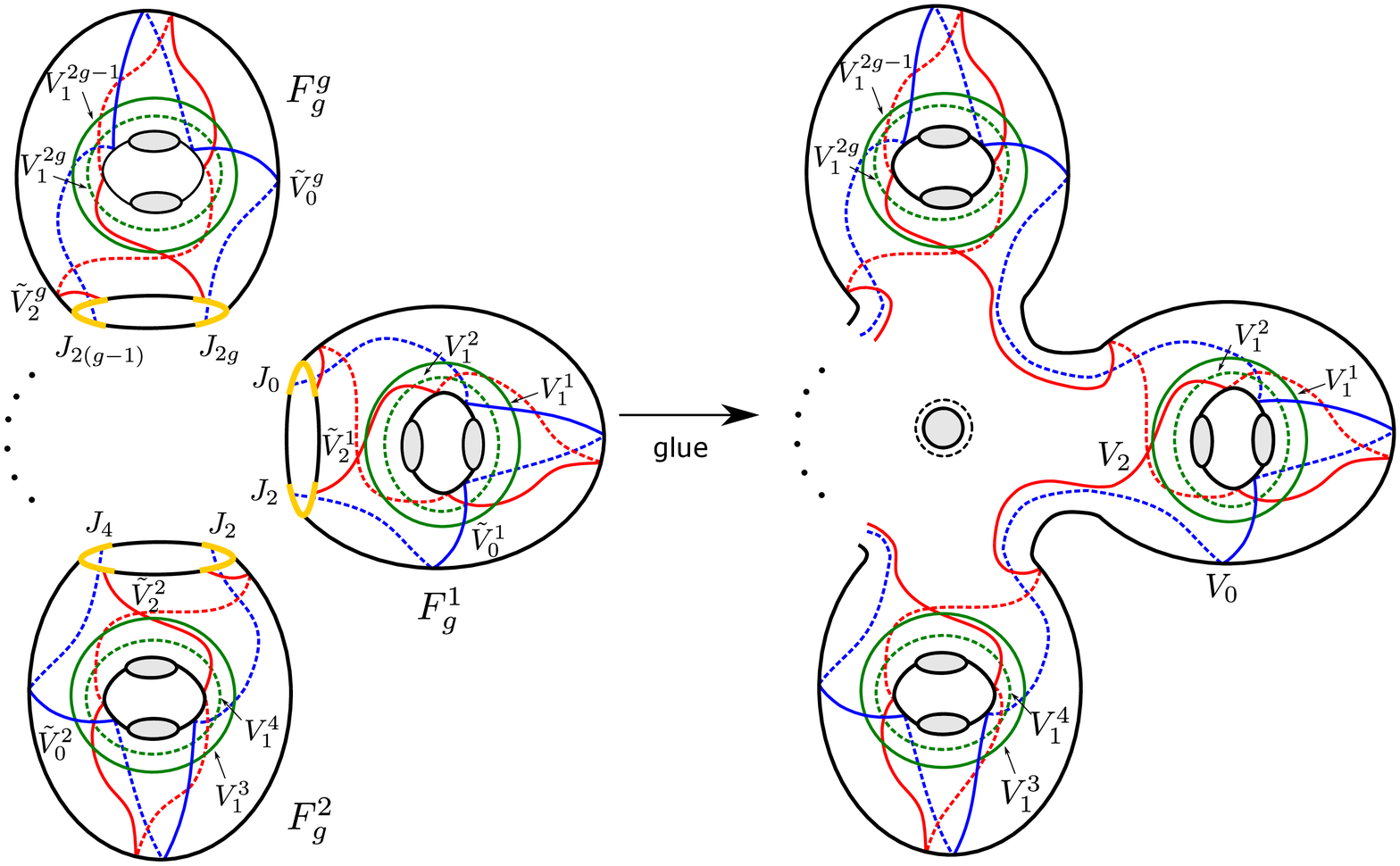}
                        \caption{The vanishing cycles $V_{0}, V_{1}^{1}, \dots, V_{1}^{2g},
    V_{2}$ on the fiber $F_g$}
                        \label{fig: vanishing cycle}
                    \end{center}
    \end{figure}

    \begin{theorem}\label{thm: ope}
Let $V_{0}, V_{1}^{1}, \dots, V_{1}^{2g},
    V_{2}$ be the simple closed
    curves  shown on the surface $\Sigma_{g, 2g+2} \cong F_g$ depicted on the right in Figure~\ref{fig: vanishing cycle} and
    let $$\varphi_g := D(V_{0})  D(V_{1}^{1})  \cdots  D(V_{1}^{2g})
D(V_{2}) \in Map(\Sigma_{g,2g+2}, \del \Sigma_{g,2g+2}).$$ Then, for
all $g \geq 1$,
    the open book $(\Sigma_{g,2g+2}, \varphi_g)$ is adapted to $(ST^{*}\Sigma_{g}, \xi_{can})$.
    \end{theorem}

    \begin{proof}
    We first note that the open book $(\Sigma_{g, 2g+2}, \varphi_g )$ induced on $\partial W_g$ by
    $$\pi_g: (W_g, \omega_g=d \alpha_g) \rightarrow
    D^{2}$$ is
    adapted to the contact $3$-manifold $(\del W_g, \ker (\alpha_{g}|_{\partial W_{g}}))$.
    According to \cite[Theorem 1.1]{Jo},
    $(W_g, \omega_g)$ is conformally exact symplectomorphic to $(DT^{*}\Sigma_{g}, \omega_{can})$, which
    is a strong symplectic filling of $(ST^{*}\Sigma_{g},
\xi_{can})$. As a consequence, $(\del W_g, \ker (\alpha_g|_{\partial
W_{g}}))$ is contactomorphic to $(ST^{*}\Sigma_{g}, \xi_{can})$, and
hence $(\Sigma_{g, 2g+2}, \varphi_g)$ is adapted to
$(ST^{*}\Sigma_{g}, \xi_{can})$.    \end{proof}

\begin{remark}  Theorem~\ref{thm: ope} also holds for $g=0$ case.  Note that $F_0 = \Sigma_{0,2}$ is nothing but  an annulus.
In this case, $V_0=V_2$ is the core circle of this annulus, and
there is no $V_1^j$.  Therefore $(ST^{*}S^2=\mathbb{RP}^3,
\xi_{can})$ has an adapted open book whose page is an annulus and
whose monodromy is the square of the positive Dehn twist along the
core circle of the annulus.\end{remark}

\subsubsection{Another open book decomposition}

In this section, we describe another open book decomposition of
$ST^*\Sigma_g$ supporting $\xi_{can}$. To motivate our discussion,
we digress here to review some open book of $ST^*\Sigma_1 \cong T^3$
supporting $\xi_{can}$ given by Van Horn-Morris \cite{vhm}. Our goal
is to compare this open book with the one
 described in Theorem~\ref{thm: ope}, for the case $g=1$. The page
of  the open book described in \cite[Chapter 6]{vhm} is
diffeomorphic to a $4$-holed torus $\Sigma_{1,4}$ and its monodromy
is given by
$$ \widetilde{\varphi}_1 :=
D^{-2}(a_{1}) D^{-2}(a_{2}) D^{-2}(a_{3})D^{-2}(a_{4})D(\delta_{1})
D(\delta_{2}) D(\delta_{3}) D(\delta_{4}),$$ where $a_{1}, a_{2},
a_{3}, a_4,  \delta_{1}, \delta_{2}, \delta_{3}, \delta_{4}$ are
shown on the $4$-holed torus $\Sigma_{1,4}$ depicted in
Figure~\ref{fig: 4-holdetorus}.

    \begin{figure}[h]
                    \begin{center}
                        \includegraphics[width=200pt]{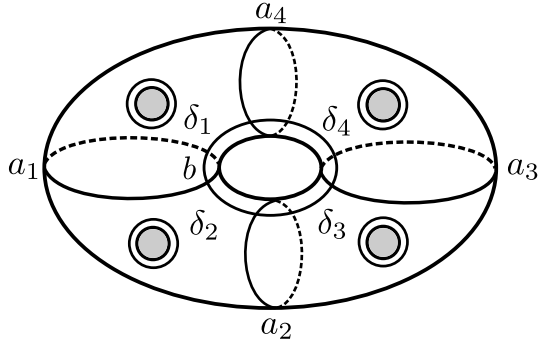}
                        \caption{$4$-holed torus $\Sigma_{1,4}$}
                        \label{fig: 4-holdetorus}
                    \end{center}
    \end{figure}

Using the relation (cf. \cite{ko}) $$D(\delta_{1}) D(\delta_{2})
D(\delta_{3}) D(\delta_{4}) = (D(a_{1})  D(a_{3}) D(b)D(a_{2})
D(a_{4})D(b))^2 \in Map(\Sigma_{1,4}, \del \Sigma_{1,4})$$ and
setting $A_{1,3} := D(a_{1})  D(a_{3})$, $A_{2,4} := D(a_{2})
D(a_{4})$, we see that $\widetilde{\varphi}_1$ is equivalent to a
product of four positive Dehn twists:
\begin{align*} &
\widetilde{\varphi}_1= A_{2,4}^{-1} A_{1,3}^{-1} A_{2,4}^{-1}
A_{1,3}^{-1} A_{1,3}
D(b)A_{2,4} D(b) A_{1,3} D(b)A_{2,4}D(b) \\
& = A_{2,4}^{-1} \big[ A_{1,3}^{-1} \big((A_{2,4}^{-1}D(b)A_{2,4})
D(b)\big) A_{1,3} D(b)\big]A_{2,4}D(b) \\
&=D(A_{2,4}^{-1}  A_{1,3}^{-1} A_{2,4}^{-1}(b)) \; D(A_{2,4}^{-1}
A_{1,3}^{-1} (b))  \; D(A_{2,4}^{-1}(b)) \; D(b)\\
&\equiv D(A_{2,4}^{-1} (b)) \; D(b)  \; D( A_{1,3}(b)) \; D(A_{1,3}
A_{2,4}(b)),
 \end{align*}
 where the notation ``$\equiv$" means ``related by a global
 conjugation". Here we conjugated with the diffeomorphism $A_{1,3}A_{2,4}$, to obtain
 the last line from the previous one. We would like to compare this open book with the one
 described in Theorem~\ref{thm: ope}, for the case $g=1$. The latter has
monodromy
$$\varphi_1 = D(V_{0})  D(V_{1}^{1}) D(V_{1}^{2}) D(V_{2}) \in Map(\Sigma_{1,4}, \del
\Sigma_{1,4}),$$ where the curves $V_{0},  V_{1}^{1},
V_{1}^{2},V_{2}$ are depicted in Figure~\ref{fig: 4-holdetorus2}.

 \begin{figure}[h]
                    \begin{center}
                        \includegraphics[width=200pt]{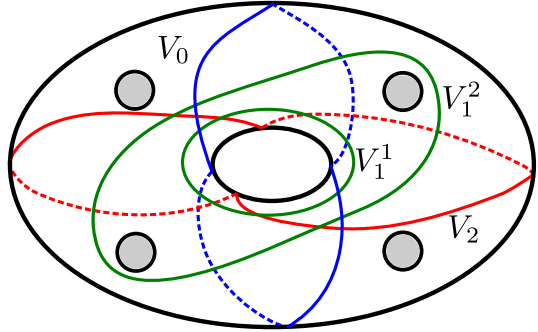}
                        \caption{The vanishing cycles $V_{0},  V_{1}^{1},
V_{1}^{2},V_{2}$ on $\Sigma_{1,4}$}
                        \label{fig: 4-holdetorus2}
                    \end{center}
    \end{figure}

Now one can easily verify that $V_0 = A_{2,4}^{-1}(b)$, $V_1^1 = b$,
$V_1^2 =A_{2,4}^{-1} A_{1,3} (b)$, and $V_2= A_{1,3}(b)$, using our
notation above.

Hence we get
$$\varphi_1 = D(A_{2,4}^{-1}  (b)) \;  D(b) \; D(A_{2,4}^{-1} A_{1,3} (b)) \; D(A_{1,3}(b)),$$
and we claim that $\varphi_1$ and $\widetilde{\varphi}_1$ are
Hurwitz-equivalent. To see this,  we apply a Hurwitz move to
$\widetilde{\varphi}_1$. Namely, we switch the order of the last two
Dehn twists in the factorization of $\widetilde{\varphi}_1$ as
follows:
$$ D(A_{2,4}^{-1} (b)) \; D(b)   \; D(A_{1,3} D(b) A^{-1}_{1,3}   A_{1,3}
A_{2,4}(b)) \; D( A_{1,3}(b)).$$ Here we used the relation
$D(A_{1,3}(b))=A_{1,3} D(b) A^{-1}_{1,3}$. To prove our claim, one
can simply verify that $A_{1,3} D(b) A^{-1}_{1,3} A_{1,3}
A_{2,4}(b)$ is isotopic to $A_{2,4}^{-1} A_{1,3} (b)$ by a direct
calculation on the surface $\Sigma_{1,4}$.

The upshot is that the open books (both of whose page is a $4$-holed
torus) given by monodromies $\varphi_1$ and $\widetilde{\varphi}_1$,
respectively,  are indeed isomorphic. Moreover, there is an exact
symplectic Lefschetz fibration
 $\widetilde{\pi}_1 : (\widetilde{W}_1, \widetilde{\omega}_1) \to D^2$ whose
monodromy is  $\widetilde{\varphi}_1$. Recall that we already
considered an exact symplectic Lefschetz fibration $\pi_1: (W_1,
\omega_1) \rightarrow D^{2}$ whose monodromy is $\varphi_1$, at the
beginning of Section~\ref{sec: exp}. By Lemma~\ref{lem: hur}, we
immediately deduce the following corollary.

\begin{corollary} The exact symplectic Lefschetz fibrations $\pi_1: (W_1, \omega_1) \rightarrow D^{2}$ and
 $\widetilde{\pi}_1 : (\widetilde{W}_1, \widetilde{\omega}_1) \to D^2$ are
isomorphic through an exact symplectomorphism.
\end{corollary}

In his thesis \cite[Chapter 4]{vhm}, Van Horn-Morris describes
another open book adapted to the contact $3$-manifold
$(ST^*\Sigma_1, \xi_{can})$, whose page is a $3$-holed torus
$\Sigma_{1,3}$ (rather than $4$-holed)  and whose monodromy is given
by
$$ \psi_1 = D(\delta_{1}) D(\delta_{2}) D(\delta_{3}) D^{-3}(a_{1})
D^{-3}(a_{2}) D^{-3}(a_{3}),$$ where $a_{1}, a_{2}, a_{3}, b,
\delta_{1}, \delta_{2}, \delta_{3}$ are shown on the $3$-holed torus
$\Sigma_{1,3}$ depicted in Figure \ref{fig: 3-holedtorus}.

 \begin{figure}[h]
                    \begin{center}
                        \includegraphics[width=150pt]{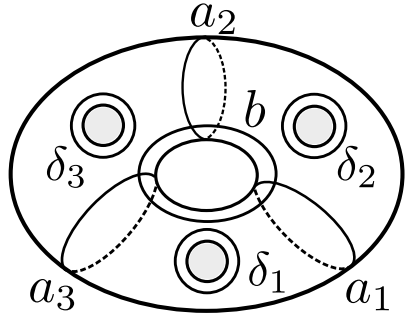}
                        \caption{$3$-holed torus $\Sigma_{1,3}$}
                        \label{fig: 3-holedtorus}
                    \end{center}
    \end{figure}

By using the star relation $D(\delta_{1})  D(\delta_{2})
D(\delta_{3}) = (D(b)  D(a_{1})  D(a_{2})  D(a_{3}))^3$, and setting
    $$T=D(a_{1})  D(a_{2})  D(a_{3}),$$ we have
    \begin{align*}
        &\psi_1= D(\delta_{1})  D(\delta_{2})  D(\delta_{3})  D^{-3} (a_{1}) D^{-3}(a_{2})  D^{-3}(a_{3}) \\
          &= (D(b) T)^3  T^{-3} \\
         &= D(b) TD(b) TD(b) T^{-2}\\
         &\equiv  (T^{-1} D(b) T)  D(b)(T D(b) T^{-1})\\
        &= D(T^{-1}(b)) \; D(b) \; D(T(b)).
    \end{align*}

Hence we see that the monodromy of this open book is equivalent to
the product of three positive Dehn twists. Moreover, since both
$(\Sigma_{1,3}, \psi_1)$ and $(\Sigma_{1,4},
\widetilde{\varphi}_1)$,
 are adapted to the contact $3$-manifold $(ST^*\Sigma_1, \xi_{can})$, they must have a common
positive stabilization. As a matter of fact, one can easily verify
that $(\Sigma_{1,3}, \psi_1)$ stabilized twice and $(\Sigma_{1,4},
\widetilde{\varphi}_1)$ stabilized once are equivalent, using the
lantern relation.


    Motivated by the genus one case, for each $g \geq 1$, we construct an open book adapted to $(ST^{*}\Sigma_{g}, \xi_{can})$ whose page is
    diffeomorphic to $\Sigma_{g, g+2}$, reducing the number of boundary components of the page, compared to that which appeared in
    Theorem~\ref{thm: ope}.
    The key idea is to cut down one boundary component for each building block that we
    used
    above to construct the page $F_g$. To construct this new open book adapted to $(ST^{*}\Sigma_{g},
\xi_{can})$,
    we introduce a new building block, inspired by the genus one case.
    We set $u_{0} := T^{-1}(b)$ and  $\ u_{2}:=T(b)$, as depicted in Figure \ref{fig:
    factorize}.

 \begin{figure}[h]
                    \begin{center}
                        \includegraphics[width=150pt]{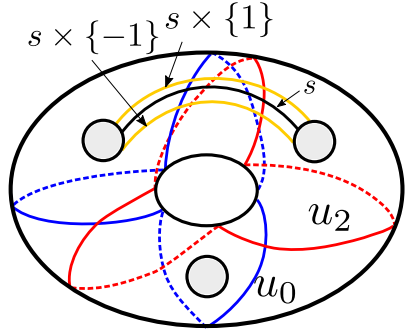}
                        \caption{The curves $u_0$, $u_2$ and the arc $s$ on $\Sigma_{1,3}$}
                        \label{fig: factorize}
                    \end{center}
    \end{figure}

    Let $s$ be the arc whose endpoints lie on two distinct boundary components in a $3$-holed torus $\Sigma_{1,3}$ as shown in Figure \ref{fig: factorize} and
    let $N(s)$ denote a tubular neighborhood of $s$.
    We write $\tilde{\Sigma}$ for the resulting surface, after removing $s\times (-1,1)$ from $\Sigma_{1,3}$, where we
    identify $N(s)$ with $s\times [-1,1]$. We write $\tilde{u}_{0}$ and $\tilde{u}_{2}$ for the two arcs in $\tilde{\Sigma}$
    obtained from the curves $u_{0}$ and $u_{2}$, respectively, by removing their intersection with $N(s)$.

    We take $g$ copies of $\tilde{\Sigma}$ and denote each copy by
    $\tilde{\Sigma}_{j}$, for $j=1,\ldots, g$. We set $$s_{j} \times \{\pm 1\} := s \times \{\pm 1\} \subset
\tilde{\Sigma}_{j},$$
    $$\tilde{u}_{i}^{j} := \tilde{u}_{i} \subset \tilde{\Sigma}_{j} {(i=0,2)}, \;
    \tilde{u}_{1}^{j} := u_{1} \subset \tilde{\Sigma}_{j}, $$
    and put $\tilde{\Sigma}_{1}, \tilde{\Sigma}_{2}, \dots, \tilde{\Sigma}_{g}$ in a circular position
    as depicted on the left in Figure \ref{fig: Xg}. Now, we glue $\tilde{\Sigma}_{j}$ to $\tilde{\Sigma}_{j+1}$ by identifying
     $s_{j}\times \{1\}$ with $s_{j+1}\times \{-1\}$ for
    $j=1,2, \dots, g-1$ and we glue $\tilde{\Sigma}_{g}$ to $\tilde{\Sigma}_{1}$ by identifying
    $s_{g} \times \{ 1 \}$ with $s_{1} \times \{ -1\}$.
    As a consequence, we obtain a surface diffeomorphic to $\Sigma_{g,
    g+2}$, which is depicted on the right in Figure \ref{fig: Xg}.
    Via the identifications above, the union of the arcs
    $\tilde{u}_{0}^{j}$ and $\tilde{u}_{2}^{j}$ form simple closed curves $U_{0}$ and $U_{2}$, respectively, in
    $\Sigma_{g,g+2}$. Considering $\tilde{\Sigma}_{j}$ as a subsurface of $\Sigma_{g,
g+2}$, we denote $u_{1}^{j}$ by $U_{1}^{j}$.

      \begin{figure}[h]
                    \begin{center}
                        \includegraphics[width=450pt]{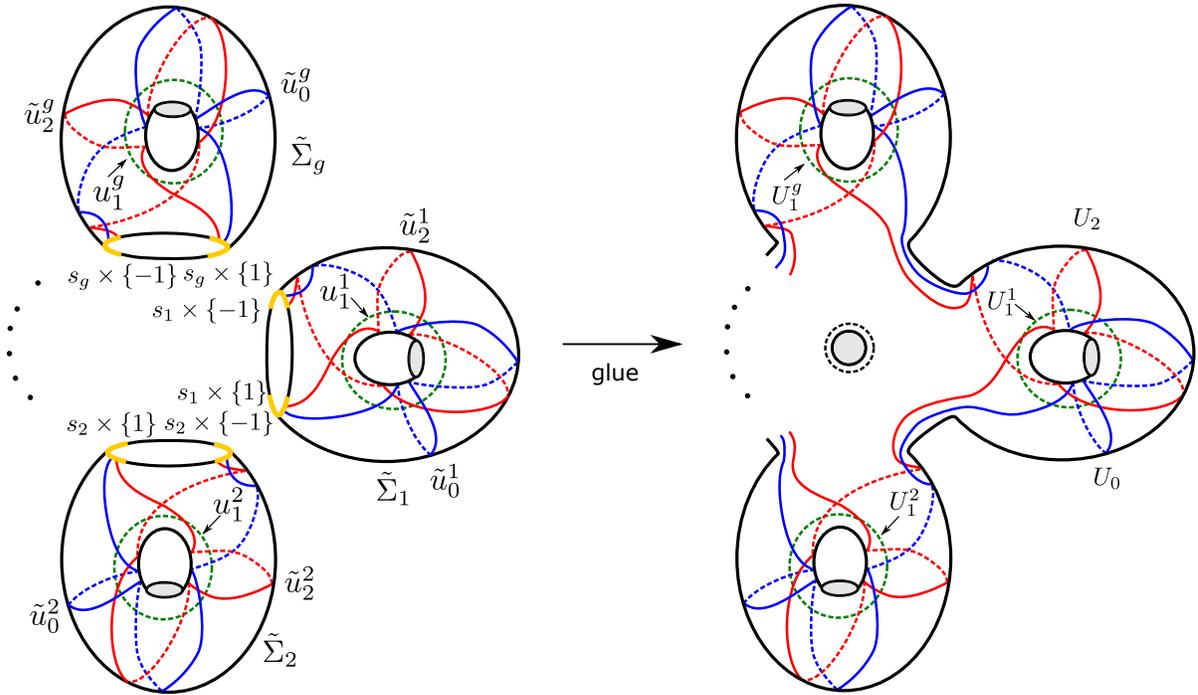}
                        \caption{The curves $U_{0}, U_{1}^{1},  \ldots,  U_{1}^{g}, U_{2}$ on the page $\Sigma_{g, g+2}$}
                        \label{fig: Xg}
                    \end{center}
    \end{figure}

    \begin{theorem}\label{theorem: g+2}  Let $U_{0},  U_{1}^{1},  \ldots,  U_{1}^{g},  U_{2}$ be the simple closed
    curves  shown on $\Sigma_{g, g+2}$  depicted on the right in Figure~\ref{fig: Xg} and let
    $$\psi_g := D(U_{0})  D(U_{1}^{1})  \cdots  D(U_{1}^{g})  D(U_{2}) \in
    Map(\Sigma_{g,g+2}, \del \Sigma_{g, g+2}).$$ Then, for all $g \geq 1$, the open book $(\Sigma_{g, g+2}, \psi_g)$ is adapted to $(ST^{*}\Sigma_{g}, \xi_{can})$.
    \end{theorem}

    \begin{proof}
   We show that $(\Sigma_{g, g+2}, \psi_g)$ and $(\Sigma_{g, 2g+2}, \varphi_g)$ have a common positive
   stabilization. The result follows from a theorem of Giroux \cite{g} coupled with our Theorem~\ref{thm: ope}.
   Let $\alpha_{1,j}, \alpha_{2,j}, \dots, \alpha_{5,j}, \beta_{j}$ (for $j=1,2,\dots, g$), and $\gamma$ be
    the simple closed curves on $\Sigma_{g, 4g+2}$ as shown in Figure \ref{fig:
    stabilize}.

 \begin{figure}[h]
                \begin{center}
                    \includegraphics[width=320pt]{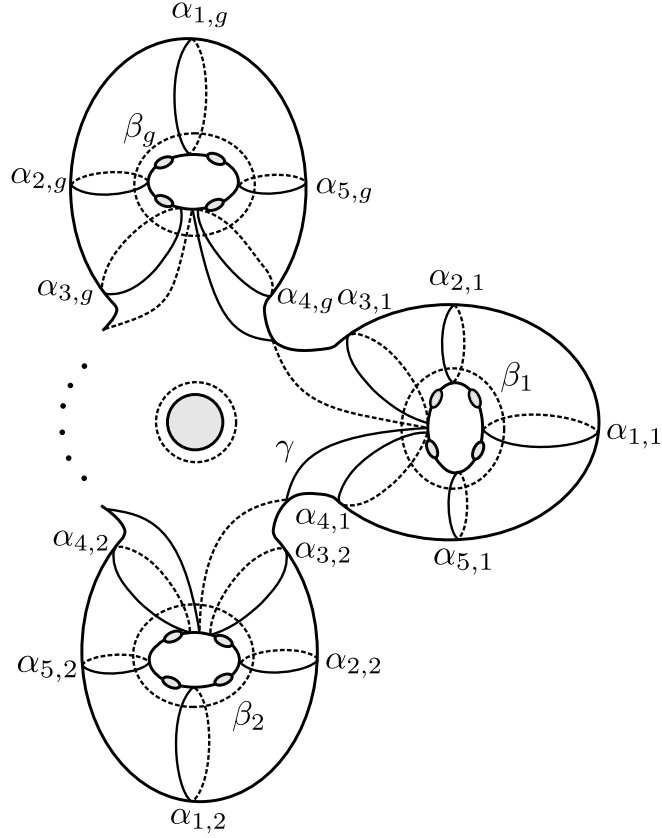}
                \caption{We illustrate a genus $g$ surface
                $\Sigma_{g, 4g+2}$  with $4g+2$ boundary components.
                There are two
                boundary components at the center: one is on top indicated by the solid circle whose interior is shaded;
                the other one is on the back side of the surface
                 indicated by the dashed circle.}
                \label{fig: stabilize}
                \end{center}
    \end{figure}
  \begin{figure}[h]
                    \begin{center}
                        \includegraphics[width=240pt]{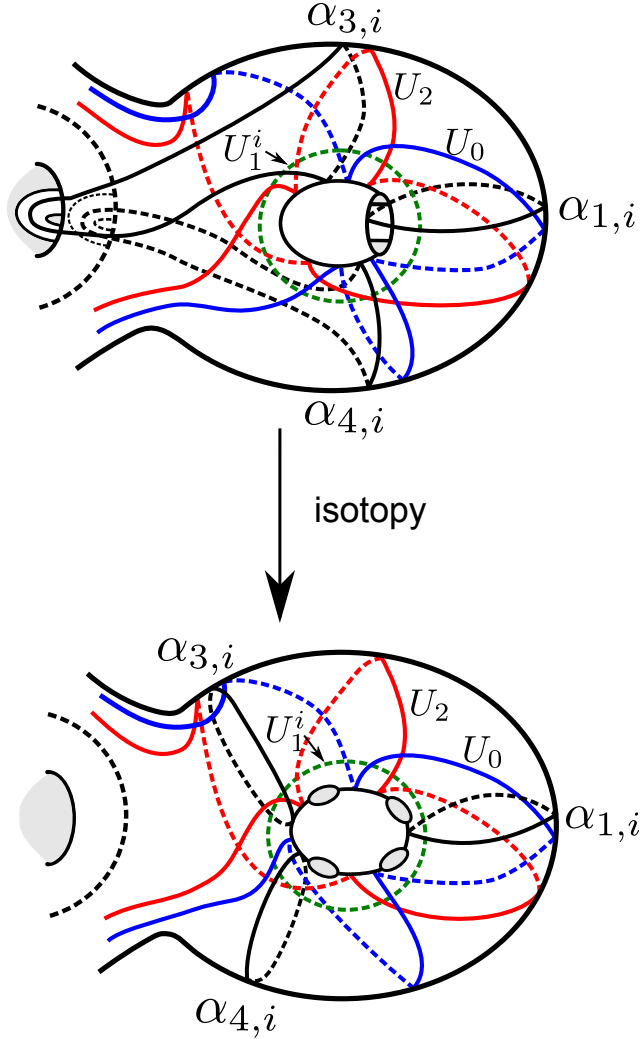}
                        \caption{Each ``building block" in the open book $(\Sigma_{g,g+2}, \psi_g)$ is stabilized three times, where the
                        stabilizing curves $\alpha_{1,i}, \alpha_{3,i}$ and $\alpha_{4,i}$ are pairwise disjoint. To go from top to bottom, we
                        just move the newly created two holes (on the top-left) by an isotopy, along with the curves $\alpha_{3,i}$ and $\alpha_{4,i}$. }
                        \label{fig: stabilize1}
                    \end{center}
    \end{figure}
    In order to prove our claim, we first stabilize the open book $(\Sigma_{g,g+2}, \psi_g)$ $3g$ times as indicated in Figure \ref{fig: stabilize1}.
    Here we just illustrate three stabilizations on each building block, where the stabilizing curves are
    $\alpha_{1,i}, \alpha_{3,i}$ and $\alpha_{4,i}$. The page of the
resulting open book is $\Sigma_{g, 4g+2}$ (identified with the
surface in Figure~\ref{fig: stabilize}) and the monodromy
{$\widehat{\psi_g} \in Map(\Sigma_{g, 4g+2}, \del \Sigma_{g,
4g+2})$} is given by
    \begin{center}
        $\widehat{\psi_g} =\psi_g D(\alpha_{1,1})  \cdots  D(\alpha_{1,g})
        D(\alpha_{3,1})  \cdots  D(\alpha_{3,g})  D(\alpha_{4,1})  \cdots  D(\alpha_{4,g})$,
    \end{center}
    where $\psi_g$ is extended to $\Sigma_{g, 4g+2}$ by identity.

\begin{figure}[h]
                    \begin{center}
                        \includegraphics[width=240pt]{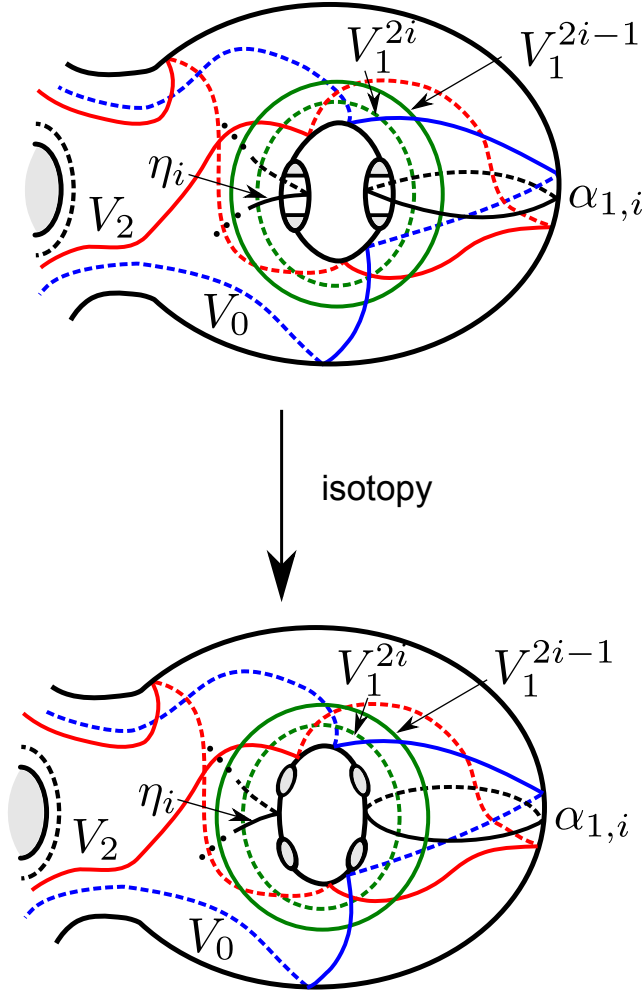}
                        \caption{Each ``building block" in the open book $(\Sigma_{g,2g+2}, \varphi_g)$ is stabilized twice, where the
                        stabilizing curves $\eta_{i} := D^{-1}(V_{2}) D^{-1}(\beta_{i}) (\alpha_{4,i})$ and $\alpha_{1,i}$
                        are disjoint. }
                        \label{fig: stabilize2}
                    \end{center}
    \end{figure}

    Similarly, we stabilize the open book $(\Sigma_{g, 2g+2}, \varphi_g)$
    $2g$
    times as indicated in Figure \ref{fig: stabilize2}. Here we just illustrate two
    stabilizations on each building block, where the stabilizing curves are
    $\alpha_{1,i}$ and $\eta_{i} := D^{-1}(V_{2}) D^{-1}(\beta_{i}) (\alpha_{4,i})$.
    The page of the resulting open book is $\Sigma_{g, 4g+2}$ (identified with the surface in Figure~\ref{fig: stabilize})
    and the monodromy $\widehat{\varphi_g} \in
Map(\Sigma_{g, 4g+2}, \del \Sigma_{g, 4g+2})$ is given by
    \begin{center}
        $\widehat{\varphi}_g =\varphi_g  D(\eta_{1})  \cdots  D(\eta_{g})  D(\alpha_{1,1})  \cdots  D(\alpha_{1,g})$,
    \end{center}
    where $\varphi_g$ is extended to $\Sigma_{g, 4g+2}$ by identity.

    Now we claim that $\widehat{\psi}_g$ and $\widehat{\varphi}_g$ are conjugate.
    First of all, both $\widehat{\psi}_g$ and $\widehat{\varphi}_g$ can be viewed as self-diffeomorphisms of the surface
    $\Sigma_{g, 4g+2}$ shown in Figure~\ref{fig:
stabilize}. In the following, we express the curves involved in the
definitions of $\widehat{\psi}_g$ and $\widehat{\varphi}_g$ in terms
of those depicted in Figure~\ref{fig: stabilize}. For convenience,
we set

$\alpha_{i} := \cup_{j=1}^{g} \alpha_{i,j},\;\;\;\;\beta:=
\cup_{j=1}^{g} \beta_{j}, \;\;\;\; D(\alpha_{i}):= \Pi_{j=1}^{g}
D(\alpha_{i,j}), \;\;\;\;D(\beta):=\Pi_{j=1}^{g} D(\beta_{j}).$

Then, we have

        $U_{0} = D^{-1} (\alpha_{2}) D^{-1} (\alpha_{5})  D(\beta)
        (\gamma), \;\;\;\;\; U_{1}^{j} = \beta_{j}, \;\;\;\;\; U_{2} = D(\alpha_{2})  D(\alpha_{5}) D^{-1} ( \beta) (\gamma),$

$V_{0} = D (\alpha_{4})  D^{-1}(\alpha_{2})  D^{-1} (\alpha_{5})
D(\beta) (\gamma)= D(\alpha_{4}) (U_{0}), \;\;\;\;\; V_{1}^{2j-1} =
D^{-1} (\alpha_{3,j}) D(\alpha_{4,j}) (\beta_{j})$,

$V_{1}^{2j} = \beta_{j},\;\;\;\;\;V_{2} = D^{-1} (\alpha_{3})
D(\alpha_{2})  D(\alpha_{5})  D^{-1} (\beta) (\gamma) =
D^{-1}(\alpha_{3})(U_{2})$.

    In the following argument,  we write $D(\alpha_{i})  D(\beta)  D^{-1} (\alpha_{i})$ for
    $\Pi_{j=1}^{g} (D(\alpha_{i, j})  D(\beta_{j})  D^{-1}(\alpha_{i, j}))$,
    and this is justified by the fact that $D(\alpha_{i, k}) D(\beta_{j}) = D(\beta_{j}) D(\alpha_{i, k})$
     for $k \neq j$.

     Using the notation ``$\equiv$" for ``related by a cyclic permutation
     or a
Hurwitz move", and underlining each pair of Dehn twists where we
perform a Hurwitz move,  we have

    \begin{align}
&\widehat{\psi}_g = \psi_g D(\alpha_{1,1})  \cdots  D(\alpha_{1,g})
        D(\alpha_{3,1})  \cdots  D(\alpha_{3,g})  D(\alpha_{4,1})  \cdots  D(\alpha_{4,g}) \notag\\
        &= D(U_{0})  D(U_{1}^{1})  \cdots  D(U_{1}^{g})  D(U_{2}) \underline{D(\alpha_{1})
        D(\alpha_{3})}  {D(\alpha_{4})} \notag\\
        &\equiv
D( D^{-1}(\alpha_{2}) D(\alpha_{5})  D(\beta) (\gamma))  D(\beta)
D(D(\alpha_{2})
        D(\alpha_{5})  D^{-1}(\beta)(\gamma))  D(\alpha_{3})
        D(\alpha_{1})  {D(\alpha_{4})} \notag\\
        &\equiv
        \underline{D(\alpha_{4})
        D(D^{-1}(\alpha_{2})  D(\alpha_{5})  D(\beta) (\gamma))}
        D(\beta)
        \underline{D( D^{-1}(\alpha_{2}) D(\alpha_{5})  D^{-1}(\beta)(\gamma))  D(\alpha_{3})}
        D(\alpha_{1}) \notag\\
        &\equiv
         D(D(\alpha_{4})  D^{-1}(\alpha_{2}) D(\alpha_{5})  D(\beta) (\gamma))
        D(\alpha_{4})
        D(\beta)
        D(\alpha_{3})
        D(D^{-1}(\alpha_{3}) D(\alpha_{2})  D(\alpha_{5})  D^{-1}(\beta)(\gamma))
        D(\alpha_{1}) \notag\\
        &=
        D(V_{0})  \underline{D(\alpha_{4})  D(\beta) } D(\alpha_{3})  D(V_{2})  D(\alpha_{1})\notag \\
        &\equiv
        D(V_{0})  D(D(\alpha_{4})(\beta))  \underline{D(\alpha_{4})  D(\alpha_{3}) }   D(V_{2})  D(\alpha_{1}) \notag\\
        &\equiv
        D(V_{0})  \underline{ D(D(\alpha_{4})(\beta))  D(\alpha_{3})}  D(\alpha_{4})    D(V_{2})  D(\alpha_{1}) \notag\\
        &\equiv
        D(V_{0})  D(\alpha_{3})  D(D^{-1}(\alpha_{3}) D(\alpha_{4}) (\beta))  D(\alpha_{4})  D(V_{2})  D(\alpha_{1})\notag \\
        &\equiv
        D(V_{0})  \underline{D(\alpha_{3})  \{ \Pi_{j=1}^{g} D(V_{1}^{2j-1}) \} }  D(\alpha_{4})  D(V_{2})  D(\alpha_{1})\notag \\
        &\equiv
        D(V_{0})   \{ \Pi_{j=1}^{g} D(V_{1}^{2j-1}) \}
        D(\{ \Pi_{j=1}^{g} D^{-1}(V_{1}^{2j-1}) \} (\alpha_{3})) D(\alpha_{4})  D(V_{2})  D(\alpha_{1}) \notag
\end{align}

Using the fact that $D^{-1}(V_{1}^{2j-1})(\alpha_{3, j})$ and
$D(\alpha_{4, j}) (\beta_{j})$ are isotopic, we continue the
sequence of equivalences above as
           \begin{align}
          &=
        D(V_{0})  \{ \Pi_{j=1}^{g} D(V_{1}^{2j-1})\}  \underline{D(D(\alpha_{4}) (\beta))  D(\alpha_{4})}  D(V_{2})
        D(\alpha_{1}) \notag\\
        &\equiv
        D(V_{0})  \{\Pi_{j=1}^{g} D(V_{1}^{2j-1})\}  \underline{D(\alpha_{4})  D(\beta)}  D(V_{2})  D(\alpha_{1}) \notag\\
        &\equiv
        D(V_{0})  \{\Pi_{j=1}^{g} D(V_{1}^{2j-1})\}  D(\beta)  \underline{D(D^{-1}(\beta)(\alpha_{4}))  D(V_{2})}  D(\alpha_{1}) \notag\\
        &\equiv
        D(V_{0}) \{ \Pi_{j=1}^{g} D(V_{1}^{2j-1})\} \{ \Pi_{j=1}^{g} D(V_{1}^{2j})\}  D(V_{2})  D(D^{-1}(V_{2}) D^{-1}(\beta)(\alpha_{4}))  D(\alpha_{1}) \notag\\
        &\equiv
        D(V_{0})  \{\Pi_{j=1}^{g} D(V_{1}^{2j-1})\} \{ \Pi_{j=1}^{g} D(V_{1}^{2j})\}  D(V_{2}) \{ \Pi_{j=1}^{g} D(\eta_{j}) \} D(\alpha_{1}) \notag\\
&=D(V_{0})  \{\Pi_{j=1}^{g} D(V_{1}^{2j-1})\} \{ \Pi_{j=1}^{g} D(V_{1}^{2j})\}  D(V_{2}) \{ \Pi_{j=1}^{g} D(\eta_{j}) \} D(\alpha_{1}) \notag\\
        &=  \varphi_g  D(\eta_{1})  \cdots  D(\eta_{g})  D(\alpha_{1,1})  \cdots  D(\alpha_{1,g}) \notag\\
&=  \widehat{\varphi}_g. \notag
    \end{align}

    Since a cyclic permutation is equivalent to a global conjugation of the monodromy, and
    a Hurwitz move does not affect the monodromy,
    we conclude that $\widehat{\psi}_g = \widehat{\varphi}_g$ up to conjugation.
    Therefore, the open books $(\Sigma_{g, g+2}, {\psi}_g)$ and $(\Sigma_{g, 2g+2}, {\varphi_g})$ have a common positive stabilization.\end{proof}

    \begin{corollary} \label{cor: elf}
     Let $\pi_g': (W_g', \omega_g') \rightarrow D^{2}$ denote the exact symplectic Lefschetz fibration, whose regular fiber is $\Sigma_{g, g+2}$ and
    whose monodromy is  $$\psi_g=D(U_{0}) D(U_{1}^{1}) \cdots D(U_{1}^{g}) D(U_{2}).$$
    Then, for all $g \geq 1$, the completion of $(W_g', \omega_g')$ is symplectomorphic to the completion
    of $(DT^{*}\Sigma_{g}, \omega_{can})$. In particular,
    $W_g'$ is diffeomorphic to $DT^{*}\Sigma_{g}$.
\end{corollary}

\begin{proof} By the proof of Theorem~\ref{theorem: g+2}, we see that $\pi_g$ (defined at the beginning of Section~\ref{sec: exp})
and $\pi_g'$ have a common positive stabilization, up to Hurwitz
moves and global conjugations. Note that a  global conjugation
induces an isomorphism of exact symplectic Lefschetz fibrations
through a symplectomorphism of their total spaces.  Therefore, the
statement follows by combining Lemma~\ref{lem: hur} and
Lemma~\ref{lem: stab}.\end{proof}

As mentioned in Section \ref{sec: intro}, according to Wendl
\cite{w}, any minimal strong symplectic filling of
$(ST^{*}\Sigma_{1} \cong T^3, \xi_{can})$ is symplectic deformation
equivalent to $(DT^{*}\Sigma_{1} \cong T^{2} \times D^2,
\omega_{can})$. Therefore, we would like to finish this section with
the following question.

\begin{question}\label{ques: sydef} Is it true that $(W_{g}', \omega_{g}')$ in Corollary \ref{cor:
elf} is symplectic deformation equivalent to $(DT^{*}\Sigma_{g},
\omega_{can})$, for all $g\geq2$?\end{question}

It is plausible that the answer to Question~\ref{ques: sydef} is
positive, via a Liouville type flow as in \cite{Jo}, although we
could not verify it.

\subsection{Unit contact cotangent bundles of non-orientable surfaces}\label{sec: nori}
In this section, we assume that $S$ is the closed
\emph{non-orientable} surface  obtained by the connected sum of $k$
copies of $\mathbb{RP}^2$, which we denote by $N_k$. We also denote
the exact symplectic Lefschetz fibration of Johns discussed above by
$\pi_k: (M_k, \omega_k) \rightarrow D^{2}$, where $(M_k, \omega_k)$
is
    conformally exact symplectomorphic to $(DT^{*}N_{k},
    \omega_{can})$. We first review the Lefschetz fibration $\pi_k$
    by describing its fiber and a set of vanishing cycles.

 The
fiber $F_{k}$ (see Figure \ref{fig: F_{k}}) is constructed as
follows: Let $R_{k}$ denote the rectangle $[0, k] \times [-1,1]$ in
$\R^{2}$ equipped with the standard orientation.

\begin{figure}[h]
                    \begin{center}
                        \includegraphics[width=250pt]{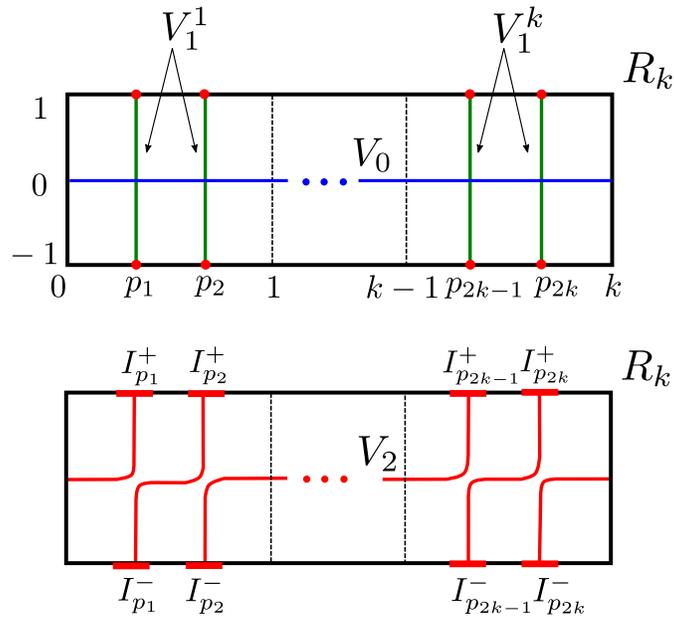}
                        \caption{The vanishing cycles $V_{0}, V_{1}^{1}, \dots, V_{1}^{k}, V_{2}$}
                        \label{fig: F_{k}}
                    \end{center}
   \end{figure}

We fix the points
$$p_{2i-1}:= (i-1) + 1/3, \  p_{2i}:= (i-1)+2/3, $$
for $i=1,2,\dots, k$ on the $x$-axis.
For a sufficiently small $\varepsilon>0$, set
$$I_{p_{j}}^{\pm}:= [p_{j}-\varepsilon, p_{j}+ \varepsilon] \times \{ \pm 1\} \subset \R^{2},$$
for $j=1,2,\dots, 2k$. We first identify $\{0\} \times [-1,1]$ with
$\{k\} \times [-1,1]$ to obtain an annulus. Next, we identify
$I_{p_{2j-1}}^{+}$ with $I_{p_{2j}}^{+}$, and $I_{p_{2j-1}}^{-}$
with $I_{p_{2j}}^{-}$ for $j=1,2,\dots, k$

 Note that, for each $j=1,2, \ldots,
    k$, these identifications can be viewed as attaching two
    $1$-handles, which is the same as plumbing an annulus as shown
    on the right in Figure~\ref{fig: plumbing}.

    It is clear (see Figure \ref{fig: vanishing cycles of pi_{k}}) that
the resulting oriented surface is a planar surface with $2k+2$
boundary components.

\begin{figure}[h]
                    \begin{center}
                        \includegraphics[width=230pt]{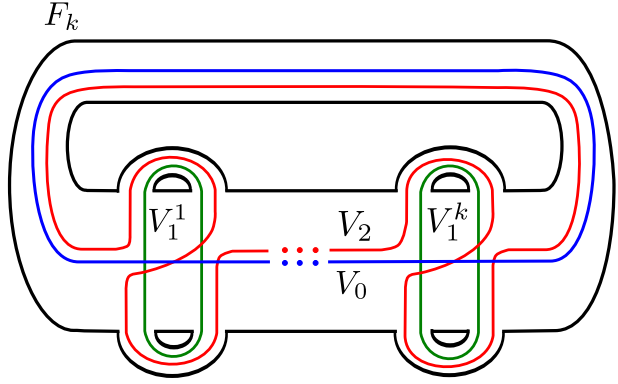}
                        \caption{The vanishing cycles $V_{0}, V_{1}^{1}, \dots, V_{1}^{k}, V_{2}$ on the fiber
                        $F_{k}$}
                        \label{fig: vanishing cycles of pi_{k}}
                    \end{center}
    \end{figure}

Now we describe the vanishing cycles $V_{0}, V_{1}^{1},
\dots, V_{1}^{k}, V_{2}$ of $\pi_{k}$ for a fixed distinguished
basis of vanishing paths. The vanishing cycle $V_{0}$ is the simple
closed curve in $F_{k}$ obtained from $[0, k] \times \{0\} \subset
R_{k}$ through the above identifications. Similarly, the simple
closed curve $V_{1}^{j} \subset F_{k}$ is obtained from
$\{p_{2j-1}\} \times [-1,1] \cup \{ p_{2j}\} \times [-1,1] \subset
R_{k}$. Equivalently, $V_{1}^{j}$ is the core circle of the
     annulus that appears in the plumbing description (see Figure~\ref{fig: vanishing cycles of pi_{k}}). The last vanishing cycle $V_{2}$ is the simple closed curve in
$F_{k}$ obtained from the Lagrangian surgery of $V_{0}$ and
$\cup_{i=1}^{k} V_{1}^{k}$.

The following theorem is proved by the same argument we used to
prove Theorem~\ref{thm: ope}.

\begin{theorem}\label{thm: ope2}
Let $V_{0}, V_{1}^{1}, \dots, V_{1}^{k}, V_{2}$ be the simple closed curves shown on the surface $\Sigma_{0, 2k+2} \cong F_{k}$
depicted in Figure \ref{fig: vanishing cycles of pi_{k}} and let
$$\phi_{k}:= D(V_{0}) D(V_{1}^{1})  \cdots D(V_{1}^{k}) D(V_{2}) \in Map(\Sigma_{0, 2k+2}, \partial \Sigma_{0, 2k+2}).$$
Then, for all $k \geq 1$, the open book $(\Sigma_{0, 2k+2},
\phi_{k})$ is adapted to $(ST^{*}N_{k}, \xi_{can})$.
\end{theorem}

\begin{remark} Note that $F_1=\Sigma_{0,4}$  is the $4$-holed sphere and the
monodromy $$\phi_1=D(V_{0}) D(V_{1}^{1}) D(V_{2})$$ of the open book
given in Theorem~\ref{thm: ope2} on $ST^*\mathbb{RP}^2 = L(4,1)$  is
equal, by the lantern relation, to the product of positive Dehn
        twists along four curves each of which is parallel to a boundary component of $\Sigma_{0,4}$.\end{remark}

\section*{Appendix: Diffeomorphism types of the total spaces of the Lefschetz
fibrations} In this appendix, we verify that the total spaces of the
Lefschetz fibrations $\pi_{g}: W_{g} \to D^2$  (see
Section~\ref{sec: ori}) and $\pi_{k}: M_{k} \to D^2$ (see
Section~\ref{sec: nori}) are diffeomorphic to $DT^{*}\Sigma_{g}$ and
$DT^{*}N_{k}$, respectively.

\subsection{Orientable case}

  We show that, for each $g \geq 1$, the $4$-manifold $W_{g}$ is diffeomorphic to  $DT^{*}\Sigma_{g}$ using
  Kirby calculus.
   There is a handle decomposition of the fiber $F_{g}$, after isotopy, as shown in Figure \ref{fig:
   handle_orientable}.

\begin{figure}[h]
                    \begin{center}
                        \includegraphics[width=240pt]{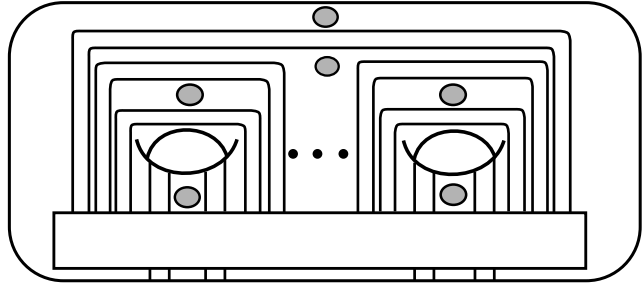}
                        \caption{Handle decomposition of $F_{g}$: The large rectangle represents the $0$-handle and each band represents a $1$-handles.}
                        \label{fig: handle_orientable}
                    \end{center}
    \end{figure}
\begin{figure}[h]
                    \begin{center}
                        \includegraphics[width=337pt]{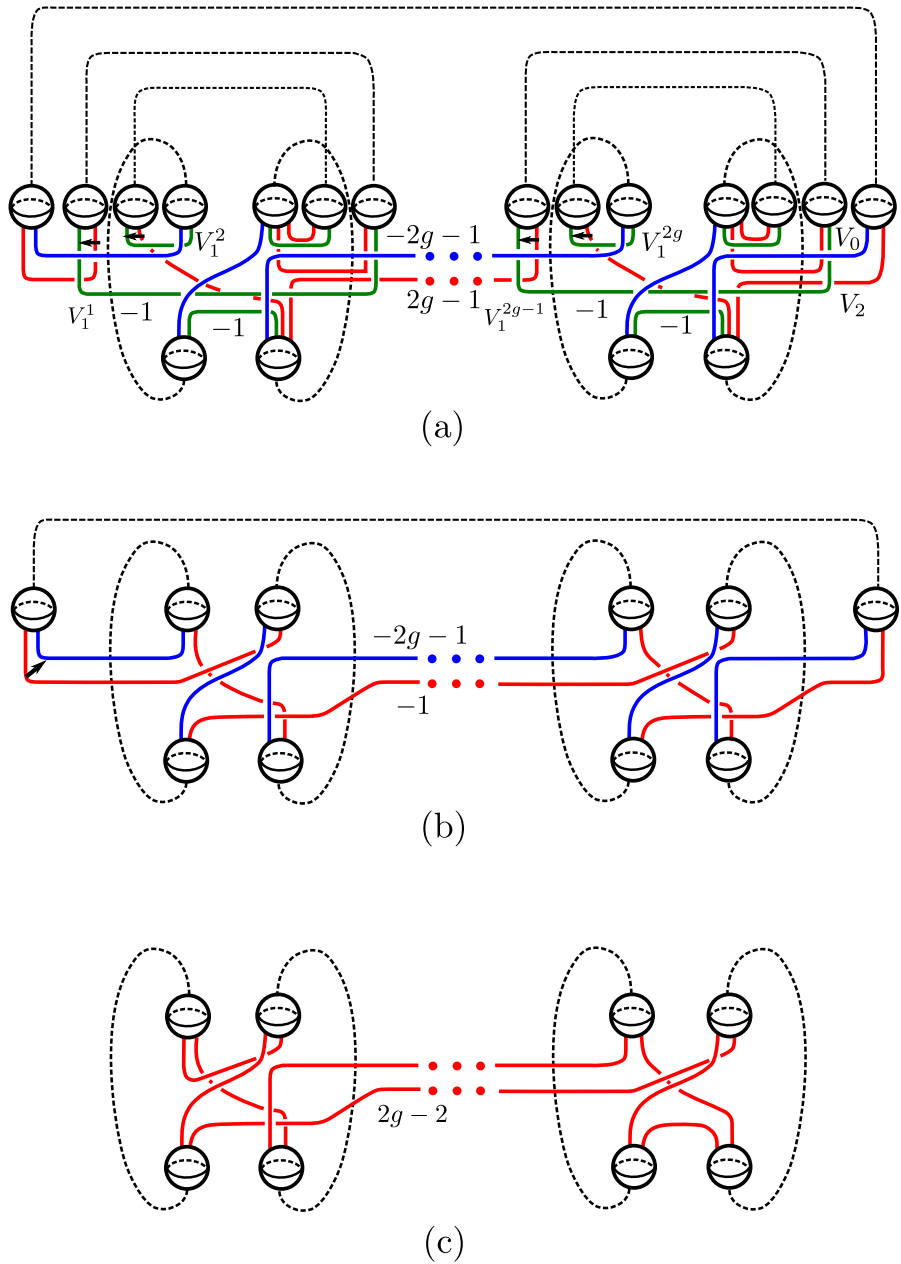}
                        \caption{Kirby diagram of $W_{g}$ and Kirby calculus: Each small arrow in the diagram indicates how we slide a $2$-handle over another one.}
                        \label{fig: Kirby_orientable}
                    \end{center}
    \end{figure}

                \begin{figure}[h]
                    \begin{center}
                        \includegraphics[width=253pt]{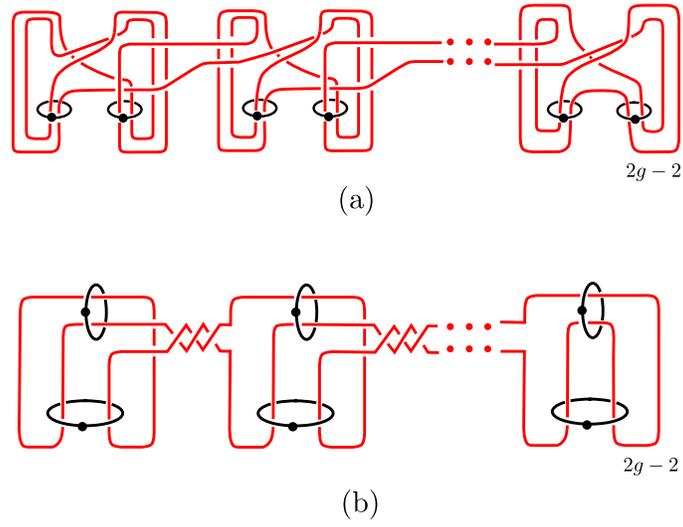}
                        \caption{Kirby diagram of $W_{g}$.}
                        \label{fig: Kirby_orientable_v2}
                    \end{center}
    \end{figure}

\begin{figure}[h]
                   \begin{center}
                        \includegraphics[width=225pt]{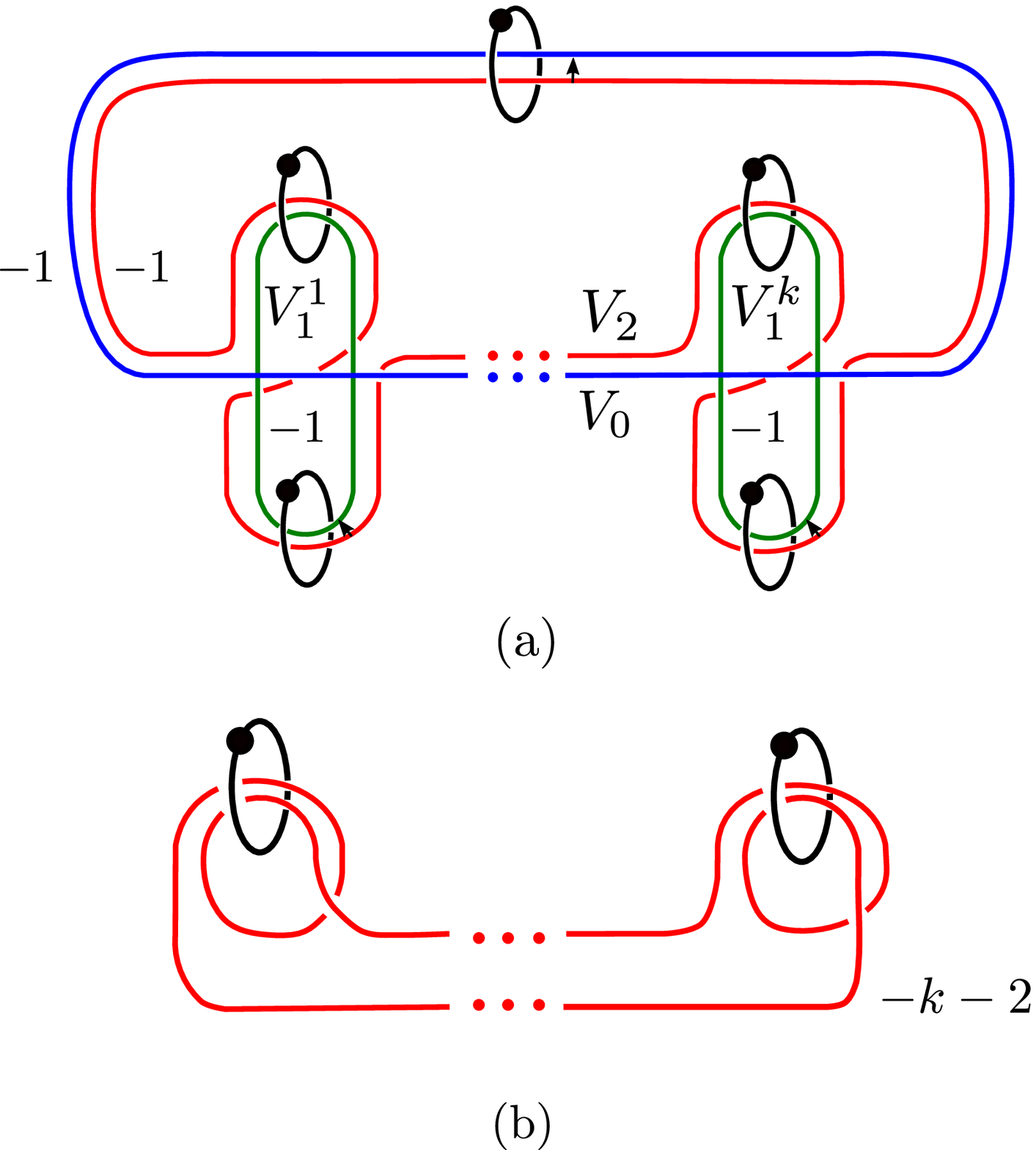}
                        \caption{Kirby diagram of $M_{k}$ and the Kirby calculus.}
                        \label{fig: Kirby_nonorientable}
                    \end{center}
    \end{figure}

    Based on this handle decomposition of $F_g$ and the collection of vanishing cycles $V_{0}, V_{1}^{1}$, $\dots$, $V_{1}^{2g}, V_{2}$,
    we draw the Kirby diagram of $W_{g}$ as depicted in Figure \ref{fig: Kirby_orientable} (a).
    Using $2$-handle slides and $1$-/$2$-handle cancelations as indicated in Figure \ref{fig: Kirby_orientable} (b),
    we obtain the Kirby diagram shown in Figure \ref{fig: Kirby_orientable} (c).  Next we switch to dotted circle notation for $1$-handles, and
    after isotopies, we see that the Kirby diagram in Figure \ref{fig:
Kirby_orientable_v2} (b) represents the disk bundle over $\Sigma_{g}$
    with Euler number $2g-2$, which is indeed diffeomorphic to
    $DT^{*}\Sigma_{g}$.

\subsection{Non-orientable case}

     We show that, for each $k \geq 1$,  the $4$-manifold $M_k$ is diffeomorphic to  $DT^{*}N_{k}$, again  using
  Kirby calculus.
    We start with the canonical handle decomposition of the fiber $F_{k}$ (see Figure \ref{fig: vanishing cycles of pi_{k}})
    and draw the Kirby diagram of $M_{k}$ as depicted in Figure \ref{fig: Kirby_nonorientable} (a).  After sliding $2$-handles and cancelling $1$-/$2$-handle pairs, we obtain the Kirby diagram shown in Figure \ref{fig: Kirby_nonorientable} (b).
    This diagram shows that $M_{k}$ is diffeomorphic to a disk bundle over $N_{k}$.
    The Euler number of this disk bundle is $k-2$ since the framing of the $2$-handle in the diagram is $-k-2$  (cf. \cite[Section
    4.6]{gs}). Therefore, we conclude that
    $M_{k}$ is diffeomorphic to $DT^{*}N_{k}$.

\noindent {\bf {Acknowledgement}}: We would like to thank Paul Seidel for helpful correspondence.
The first author would like to
express his gratitude to  Ko\c{c} University for its hospitality
during a visit while this work was mainly carried out.


\end{document}